%% file: rigid-paper.tex
\documentclass[a4paper]{article}

\usepackage{authblk}

\RequirePackage{silence}
\usepackage[utf8]{inputenc}

\usepackage[UKenglish]{babel}
\usepackage[style=alphabetic,maxbibnames=99]{biblatex}
\usepackage{amsmath}
\usepackage{mathtools}
\usepackage{amsthm}
\usepackage{thmtools}
\usepackage{csquotes}
\usepackage[inline]{enumitem}
\usepackage{nicefrac}
\usepackage{tikz}
\usepackage[hidelinks]{hyperref}
\usepackage[capitalize,noabbrev]{cleveref}
\usepackage{newunicodechar}
\usepackage[section]{placeins}

\usepackage[final]{microtype}

\usepackage[status=final]{fixme}

\usepackage[charter]{mathdesign}
\usepackage[T1]{fontenc}

\input{style}
\input{commands}

\title{Uniform, rigid branchwise-real trees}
\author[*]{Sam Adam-Day}
\affil[*]{Mathematical Institute, University of Oxford, Andrew Wiles Building, Radcliffe Observatory Quarter, Woodstock Road, Oxford, OX2 6GG, United Kingdom; \href{mailto:adamday@maths.ox.ac.uk}{\nolinkurl{adamday@maths.ox.ac.uk}}}
\date{\today}

\begin{document}

	\maketitle

	\abstract{A branchwise-real tree is a partial order which is a tree and in which every branch is isomorphic to a real interval. I give constructions of such trees which are both rigid (i.e. without non-trivial order-automorphisms) and uniform (in two different senses). Specifically, I show that there is a rigid branchwise-real tree in which every branching point has the same degree, one in which every point is branching and of the same degree, and finally one in which every point is branching of the same degree and which admits no order-preserving function into the reals. Trees are grown iteratively in stages, and a key technique is the construction (in \ZFC) of a family of colourings of $(0,\infty)$ which is `sufficiently generic', using these colourings to determine how to proceed with the construction.}

	\renewcommand{\thefootnote}{}
	\footnote{\emph{Keywords}: branchwise-real tree, tree automorphism, rigid, uniform, forcing, R-tree}
	\footnote{\emph{2020 Mathematics Subject Classification}: 03E04, 06A07, 20E08, 54F05}
	\renewcommand{\thefootnote}{\arabic{footnote}}
	\addtocounter{footnote}{-2}

	\input{parts/1-intro}
	\input{parts/2-background}
	\input{parts/3-growing-brtos}
	\input{parts/4-weakly-uniform}
	\input{parts/5-strongly-uniform}
	\input{parts/6-conclusion}
	\input{parts/7-acknowledgements.tex}

	\printbibliography

\end{document}

%% file: style.tex
\newtheorem{theorem}{Theorem}
\newtheorem{lemma}[theorem]{Lemma}
\newtheorem{proposition}[theorem]{Proposition}

\newtheorem{question}[theorem]{Question}

\theoremstyle{definition}
\newtheorem{definition}[theorem]{Definition}

\theoremstyle{remark}
\newtheorem{remark}[theorem]{Remark}

\usetikzlibrary{
	cd,
	positioning,
	patterns,
	arrows,
 	arrows.meta,
	calc,
	external,
	math,
	decorations.pathmorphing,
	decorations.markings,
	calligraphy
}
\tikzset{
	point/.style={draw=black,fill=black,opacity=1,circle,outer sep=0pt,inner sep=0,minimum size=2},
	dot/.style={draw=black,fill=black,opacity=0,circle,outer sep=0pt,inner sep=0}
}
\tikzset{
	negated/.style={
        decoration={markings,
            mark= at position 0.5 with {
                \node[transform shape] (tempnode) {$\backslash$};
            }
        },
        postaction={decorate}
    }
}
\tikzset{
	brace/.style={decoration={calligraphic brace,amplitude=5pt}, decorate, line width=1.25pt}
}

\definecolor{pallette_orange}{RGB}{230, 159, 0}
\definecolor{pallette_skyblue}{RGB}{86, 180, 233}
\definecolor{pallette_bluishgreen}{RGB}{0, 158, 115}
\definecolor{pallette_yellow}{RGB}{240, 228, 66}
\definecolor{pallette_blue}{RGB}{0, 114, 178}
\definecolor{pallette_vermillion}{RGB}{213, 94, 0}
\definecolor{pallette_reddishpurple}{RGB}{204, 121, 167}

\definecolor{pallette0}{RGB}{0, 0, 0}
\definecolor{pallette1}{RGB}{230, 159, 0}
\definecolor{pallette2}{RGB}{86, 180, 233}
\definecolor{pallette3}{RGB}{0, 158, 115}
\definecolor{pallette4}{RGB}{240, 228, 66}
\definecolor{pallette5}{RGB}{0, 114, 178}
\definecolor{pallette6}{RGB}{213, 94, 0}
\definecolor{pallette7}{RGB}{204, 121, 167}

\colorlet{seq0}{black}
\definecolor{seq1}{HTML}{DC3220}
\colorlet{seq2}{pallette_blue}
\colorlet{seq3}{pallette_orange}
\colorlet{seq4}{pallette_skyblue}
\colorlet{seq5}{pallette_bluishgreen}
\colorlet{seq6}{pallette_yellow}
\definecolor{seq7}{HTML}{0C7BDC}
\definecolor{seq8}{HTML}{994F00}
\colorlet{seq9}{black!60!white}
\definecolor{seq10}{HTML}{1AFF1A}
\definecolor{seq11}{HTML}{4B0092}

\newunicodechar{Λ}{\ensuremath{\Lam}}
\newunicodechar{ω}{\ensuremath{\omega}}

%% file: commands.tex
\DeclareMathOperator{\Aut}{\mathrm{Aut}}

\newcommand{\cont}{\mathfrak{c}}
\newcommand{\ZFC}{\ensuremath{\mathrm{ZFC}}}
\newcommand{\restr}{\mathord{\upharpoonright}}
\newcommand{\Shift}{\mathrm{Shift}}
\newcommand{\In}{\mathrm{In}}

\let\P\undefined\DeclareMathOperator{\P}{\mathcal{P}}

\DeclareMathOperator{\dom}{\mathrm{dom}}

\DeclareMathOperator{\rank}{\mathrm{rank}}

\newcommand{\R}{\ensuremath{\mathbb{R}}}
\newcommand{\Q}{\ensuremath{\mathbb{Q}}}
\newcommand{\Z}{\ensuremath{\mathbb{Z}}}

\newcommand{\A}{\ensuremath{\mathcal{A}}}

\makeatletter
\newcommand{\@usstar}[1]{{\ua}\left(#1\right)}
\newcommand{\@usnostar}[1]{{\ua}(#1)}
\newcommand{\us}{\@ifstar{\@usstar}{\@usnostar}}
\newcommand{\@dsstar}[1]{{\da}\left(#1\right)}
\newcommand{\@dsnostar}[1]{{\da}(#1)}
\newcommand{\ds}{\@ifstar{\@dsstar}{\@dsnostar}}
\newcommand{\@udsstar}[1]{{\uda}\left(#1\right)}
\newcommand{\@udsnostar}[1]{{\uda}(#1)}
\newcommand{\uds}{\@ifstar{\@udsstar}{\@udsnostar}}
\newcommand{\@Usstar}[1]{{\Ua}\left(#1\right)}
\newcommand{\@Usnostar}[1]{{\Ua}(#1)}
\newcommand{\Us}{\@ifstar{\@Usstar}{\@Usnostar}}
\newcommand{\@Dsstar}[1]{{\Da}\left(#1\right)}
\newcommand{\@Dsnostar}[1]{{\Da}(#1)}
\newcommand{\Ds}{\@ifstar{\@Dsstar}{\@Dsnostar}}
\newcommand{\@Udsstar}[1]{{\Uda}\left(#1\right)}
\newcommand{\@Udsnostar}[1]{{\Uda}(#1)}
\newcommand{\Uds}{\@ifstar{\@Udsstar}{\@Udsnostar}}
\makeatother

\newcommand{\ep}{\epsilon}
\newcommand{\lam}{\lambda}
\newcommand{\Lam}{\ensuremath\Lambda}

\newcommand{\Lra}{\Leftrightarrow}
\newcommand{\ua}{\uparrow}
\newcommand{\da}{\downarrow}
\newcommand{\uda}{\updownarrow}
\newcommand{\Ua}{\Uparrow}
\newcommand{\Da}{\Downarrow}
\newcommand{\Uda}{\Updownarrow}
\newcommand{\es}{\varnothing}
\newcommand{\sse}{\subseteq}

\newcommand{\partto}{\rightharpoonup}

\newcommand{\ol}{\overline}

\newcommand{\bb}{\mathbb}

\newcommand{\defeq}{\vcentcolon=}

\DeclarePairedDelimiter{\ab}{\langle}{\rangle}
\DeclarePairedDelimiter{\abs}{|}{|}

\newcommand{\contradiction}{\noindent
	\begin{tikzpicture}[x=0.4ex,y=0.4ex]
		\draw[line width=.15ex] (0,0) -- (1,2) -- (0,2) -- (1,4)
		(0.95,0.32) -- (0,0) -- (-0.32,0.95);
	\end{tikzpicture}\hspace*{0.2em}
}

\renewcommand{\leq}{\leqslant}
\renewcommand{\geq}{\geqslant}

%% file: parts/1-intro.tex

\section{Introduction}

This article concerns the construction of branchwise-real trees which are rigid — i.e.\@ without non-trivial order-automorphisms. Branchwise-real trees were first defined in \cite{favre2002valuative} and further elaborated on in \cite{gradability-paper-published}.\footnote{In \cite{favre2002valuative} branchwise-real trees are called `non-metric trees'.} They are tree partial orders in which every branch is order-isomorphic to a real interval, and in which every two elements have a meet (see \cref{sec:background} for definitions).

Branchwise-real trees were originally motivated from the study of \R-trees. Informally, an \R-tree is a metric space tree in which every point is permitted to be branching. They play a role in geometric group theory, in which one considers some infinite group acting by isometries on an \R-tree. For more information, see \cite{bestvina1997real,mayer92} and the references contained in \cite{Fabel15}. Now, fixing any point $p$ in an \R-tree, one can consider the \emph{cut-point order}: the set of paths through the tree from $p$, ordered by path extension. The resulting partial order is a branchwise-real tree \cite[p.~50]{favre2002valuative}. Conversely, any branchwise-real tree which admits an order-preserving function into the reals is the cut-point order of some \R-tree \cite{favre2002valuative,gradability-paper-published}. The present study of the order-automorphisms of branchwise-real trees is thus motivated by the study of the isometries of \R-trees.

A different motivation comes the investigation of the automorphisms of well-stratified trees (tree partial orders in which every branch is isomorphic to an ordinal) \cite{gaifmanspecker64,jensen69,10.2307/1996262,avraham79,Abraham1985IsomorphismTO,10.2307/40378073}. If $X$ is a well-stratified tree without uncountable branches, we can turn it into a branchwise-real tree by taking its `road space', essentially replacing every node with a copy of the interval $[0,1)$, so that every branch becomes isomorphic to a real interval \cite{REMARKSONTHENORMALMOORESPACEMETRIZATIONPROBLEM,gradability-paper-published}. Every automorphism of $X$ can then be extended to an automorphism of its road space (but not vice versa). 

However, the class of branchwise-real trees is more general than that obtained by taking the road spaces of well-stratified trees without uncountable branches. This is because any point can be branching, so that the structure of the branching points can be far from well-stratified. Consider for instance the `comb' structure given by taking a `shaft' $[0,1]$ and adding a `tooth' — a copy of $(0,1]$ — as a new branch from each point on the shaft. See \cref{fig:comb} for a picture. Every point along the shaft except the maximal one is branching of degree $2$. Moreover, the task of constructing a rigid branchwise-real tree is more challenging than in the well-stratified case. Indeed, any part of the tree which looks like a real interval and which contains no branching points allows for an easy automorphism which fixes the rest of the tree and permutes the interval in some non-trivial way.

\begin{figure}[hb]
    \centering
    \begin{tikzpicture}
        \draw[ultra thick] (0,0) -- node[above, sloped] {shaft} (0,4);
        \draw (0,0) -- node[above, sloped] {tooth} ++(2,2);
        \node at (0.8,2) {$\vdots$};
        \draw (0,2.) -- node[above, sloped] {tooth} ++(2,2);
        \node at (0.8,4) {$\vdots$};
        \draw (0,4) -- node[above, sloped] {tooth} ++(2,2);
    \end{tikzpicture}
    \caption{A representation of the comb branchwise-real tree, consisting of the shaft and a few example teeth}
    \label{fig:comb}
\end{figure}
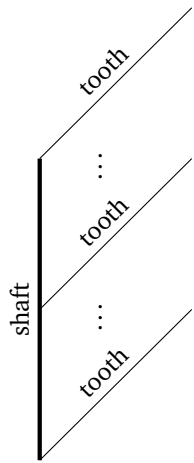

In this article, I construct examples of rigid branchwise-real trees subject to two uniformity conditions, the latter stronger than the former. First, I construct one in which every branching node has the same degree $\kappa$, for $2 \leq \kappa \leq \cont$. I call this condition `weakly uniformly $\kappa$-branching'. Second, using a different technique, I construct one in which every node is branching and of the same degree $\kappa$, for $2 \leq \kappa \leq \cont^+$. This condition I call `uniformly $\kappa$-branching'. This second tree admits an order-preserving function into the reals, and thus corresponds to the cut-point order of some \R-tree. The technique can be further adapted to produce such a uniformly $\kappa$-branching, rigid branchwise real tree which does not admit an order-preserving function into the reals, and thus is not related to an \R-tree. These results manifest the phenomenon of the marriage of two seemingly opposing notions: rigidity and uniformity. As such they continue in the vein of similar results found for well-stratified trees \cite{avraham79, 10.2307/40378073}.

All trees here are grown recursively in stages. We start with a single point: the root. At successor stages, we choose some number of points added at the previous stage, and add $(\kappa-1)$ many new `spines' — copies of the interval $(0,\infty)$ — above. (Throughout $\kappa-1$ denotes $n-1$ if $\kappa =n$ finite and $\kappa$ otherwise). We never extend directly above a spine once it has been added. At limit stages there may be new branches appearing through the tree, consisting of portions added at each of the previous stages. We decide which of these to extend by adding a new point on top.

A common technique used throughout when growing rigid trees is to first produce a collection of colourings of the interval $(0, \infty)$. As the tree grows, we take for each new spine added a different, unused colouring and lay it along the spine. In this way, we also build up a colouring of the whole tree as we go. These colourings are then used to determine which points will form the bases of new spines at successor stages, and which new branches to extend at limit stages. The collection of colourings of $(0, \infty)$ is carefully constructed so as to enable the desired properties of the final tree.

The constructions of the trees in this article, though they occur in \ZFC, have strong forcing flavour to them. Indeed, I first built these rigid trees using forcing notions, before realising that related constructions could be carried out in \ZFC. I will briefly elaborate on two of these forcing notions. For more details on forcing, see \cite{comb-set-theory-2017} and \cite{jech,kunen}. However, knowledge of forcing is not required to understand any of the \ZFC\ constructions in this article.

%% file: parts/2-background.tex

\section{Background}
\label{sec:background}

Let me begin by fixing some terminology and notation relating to partial orders. Let $P$ be a partial order. A \emph{chain} in $P$ is a linearly-ordered subset. A \emph{branch} through $P$ is a maximal chain. A \emph{ray} is a final segment of a branch. A subset $Q \sse P$ is \emph{coinitial} if for every $x \in P$ there is $y \in Q$ such that $y \leq x$. Elements $x,y \in P$ are \emph{comparable} if $x \leq y$ or $y \leq x$. Two subsets $Q,R \sse P$ are \emph{comparable} if there is $x \in Q$ and $y \in R$ such that $x$ and $y$ are comparable. An \emph{antichain} is a set of pairwise incomparable elements. For any $x \in P$ let:
\begin{equation*}
	\ds x \defeq \{y \in X \mid y \leq x\}, \qquad \us x \defeq \{y \in P \mid y \geq x\}
\end{equation*}
For any $x < y$ in $P$, define:
\begin{equation*}
	[x,y] \defeq \{z \in P \mid x \leq z \leq y\}
\end{equation*}
Define the other intervals $[x,y)$, $(x,y]$ and $(x,y)$ analogously. A function $f \colon P \to Q$ between partial orders is \emph{order-preserving} if whenever $x < y$ we have $f(x) < f(y)$. I will also call such a function a \emph{$Q$-grading} of $P$. An \emph{isomorphism} between $P$ and $Q$ is a bijection $f \colon P \to Q$ which is order-preserving with order-preserving inverse.

I can now introduce the main object of study: branchwise-real trees. I follow the presentation given in \cite{gradability-paper-published}.

\begin{definition}\label{def:tree order}
	A \emph{tree order} is a partial order $X$ such that the following conditions hold.
	\begin{enumerate}[label=(TO\arabic*), leftmargin=*, labelindent=3pt]
		\item\label{item:lo; def:tree order}
			For every $x \in X$ the set $\ds x$ is a linear order.
		\item\label{item:root; def:tree order}
			$X$ has a minimum element, its \emph{root}.
	\end{enumerate}
\end{definition}

\pagebreak

\begin{definition}\label{def:brot}
	A \emph{branchwise-real tree} is a tree order $X$ subject to the following extra conditions.\footnote{I make a slight change of notation compared with \cite{gradability-paper-published}, where these objects are referred to as `branchwise-real tree orders'.}
	\begin{enumerate}[label=(BR\arabic*), leftmargin=*, labelindent=3pt]
		\item\label{item:interval; def:brot}
			Every branch is order-isomorphic to a real interval.
		\item\label{item:meet-semilattice; def:brot}
			For any $x, y \in X$, the set $\{z \in X \mid z \leq x,y\}$ has a maximum element $x \wedge y$, the \emph{meet} of $x$ and $y$ (in other words, $X$ is a meet-semilattice).
	\end{enumerate}
\end{definition}

We will also meet trees in which every branch is well-ordered. Such trees will be called `well-stratified trees'; in a purely set theoretic context, we would simply say `trees'.

\begin{definition}
	A tree order $T$ is \emph{well-stratified} if every branch is order-isomorphic to an ordinal. The \emph{rank} of an element $x \in T$ is the order type of $\ds x \setminus \{x\}$.
\end{definition}

In order to define the \emph{degree} of a point in a tree order we need the notion of a connected component above that point.

\begin{definition}
	Let $X$ be a tree order and $x \in X$. A \emph{connected component above $x$} is an equivalence class of $\us x \setminus \{x\}$ under the relation:
	\begin{equation*}
		y \sim_x z \quad\Lra\quad \text{there is } w > x \text{ such that } w \leq y,z
	\end{equation*}
	Note that, when $X$ is a meet-semilattice, we have $y \sim_x z$ if and only if $y \wedge z > x$.	See \cref{fig:connected component relation} for an example illustration of this relation. I will usually drop the `connected' and refer to these as `components above $x$'. The \emph{degree}, $\deg(x)$, of $x$ is the number of components above $x$. Say that $x$ is \emph{terminal} if $\deg(x) = 0$. Say that $x$ is \emph{branching} if $\deg(x)>1$. The \emph{degree} of $X$ is the supremum of the degrees of its elements.
\end{definition}

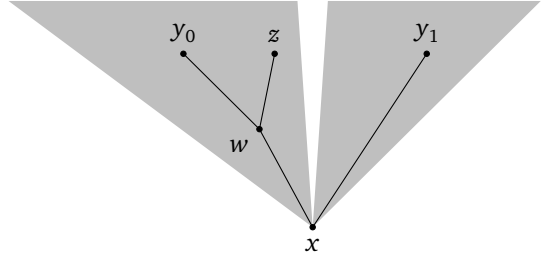
\begin{figure}[ht]
	\centering
	\begin{tikzpicture}
		\fill[lightgray] (-4,3) -- (0,0) -- (-0.2,3);
		\draw (0,0) -- (-0.7,1.3) node[point, label=below left:$w$] (w) {};
		\draw (w) -- ++(-1,1) node[point, label=above:$y_0$] (y0) {};
		\draw (w) -- ++(0.2,1) node[point, label=above:$z$] (z) {};

		\fill[lightgray] (0.2,3) -- (0,0) -- (3,3);
		\draw (0,0) -- (1.5,2.3) node[point, label=above:$y_1$] (y1) {};

		\node [point] (x) at (0,0) [label=below:$x$] {};
	\end{tikzpicture}
	\caption{An example illustrating the relation $\sim_x$ which defines components above $x$. We have $y_0 \sim_x z$ because $x < w \leq y_0, z$, but $z \not\sim_x y_1$ because there is no element above $x$ which lies below both $y_1$ and $z$.}
	\label{fig:connected component relation}
\end{figure}

The notion of \emph{continuous gradability} is important for branchwise-real trees and is the main property studied in \cite{gradability-paper-published}.

\begin{definition}\label{def:continuous grading}
	Let $X$ be a branchwise-real tree. An \R-grading $\ell \colon X \to \R$ is \emph{continuous} if for any $x < y$ in $X$ the restriction: 
	\begin{equation*}
		\ell\restr{[x,y]} \colon [x,y] \to [\ell(x),\ell(y)]
	\end{equation*}
	is an order-isomorphism. I will usually drop the `\R' and call such functions `continuous gradings'. Say that $X$ is \emph{continuously gradable} if it admits a continuous grading. 
\end{definition}

\begin{theorem}\label{thm:R grading to continuous grading}
	A branchwise-real tree is continuously gradable if and only if it is \R-gradable.
\end{theorem}

\begin{proof}[Proof sketch]
	See \cite[Theorem~30]{gradability-paper-published}. An arbitrary \R-grading $f$ of a branchwise-real tree $X$ may contain a number of discontinuities. Using a basic result from real analysis, it can be shown that such discontinuities must take the form of `jumps' in value, and that any branch may contain only countably many. The proof proceeds to eliminate every jump in $f$, using a Zorn's Lemma style argument.
\end{proof}

Two final pieces of notation. Throughout, tuples will be denoted using angle brackets: $\ab{a,b, \ldots}$. A partial function between sets $X$ and $Y$ will be denoted using $f \colon X \partto Y$.

%% file: parts/3-growing-brtos.tex

\section{How to grow branchwise-real trees}

\label{sec:growing brtos}

The trees in this article are grown using an iterative process. We always start with a singleton as the `root'. At a successor step, above each point introduced in the previous stage we can add any number of new `spines'. A spine is a copy of the positive real numbers $(0,\infty)$, or the interval $(0,\infty]$, which is `terminal', in the sense that we will never extend the tree above the end of a spine. In other words, in the final tree each spine will be a ray. Call the former type of spine an \emph{open spine} and the latter a \emph{closed spine}. At a limit step, we first take the union of the previous stages. There is more to do however, since although we do not extend above spines, new branches through the tree appear at the limit, which we may choose to extend or not. A newly appearing branch contains the root, follows a stage-$1$ spine partway, then branches onto a stage-$2$ spine, and so on, cofinally in the limit. Such a branch thus consists of a little piece from each previous stage of the construction, with no final piece. If we decide to extend it, we add a new point directly above. This process is illustrated in \cref{fig:growing tree}.

\begin{figure}[ht]
	\centering
	\begin{tikzpicture}[yscale=0.8]
		\node[point, seq0] at (0,0) [label=below:{root}] (r) {};

		\begin{scope}[every path/.style={thick}]
			\draw[seq1] (r)  -- node[dot, pos=0.500] (a1) {} ++(-3,1);
			\draw[seq2] (a1) -- node[dot, pos=0.333] (a2) {} ++(1,2);
			\draw[seq3] (a2) -- node[dot, pos=0.250] (a3) {} ++(-1,2);
			\draw[seq4] (a3) -- node[dot, pos=0.200] (a4) {} ++(1,2);
			\draw[seq5] (a4) -- node[dot, pos=0.167] (a5) {} ++(-1,2);
			\draw[seq6] (a5) -- node[dot, pos=0.143] (a6) {} ++(1,2);
			\draw[seq7] (a6) -- node[dot, pos=0.125] (a7) {} ++(-1,2);
			\draw[seq8] (a7) -- node[dot, pos=0.111] (a8) {} ++(1,2);
		\end{scope}
		\node [above=10pt of a7] (ad) {$\vdots$};

		\begin{scope}[every path/.style={thick}]
			\draw[seq1] (r)  -- node[dot, pos=0.500] (b1) {} ++(3,1);
			\draw[seq2] (b1) -- node[dot, pos=0.333] (b2) {} ++(-1,2);
			\draw[seq3] (b2) -- node[dot, pos=0.250] (b3) {} ++(1,2);
			\draw[seq4] (b3) -- node[dot, pos=0.200] (b4) {} ++(-1,2);
			\draw[seq5] (b4) -- node[dot, pos=0.167] (b5) {} ++(1,2);
			\draw[seq6] (b5) -- node[dot, pos=0.143] (b6) {} ++(-1,2);
			\draw[seq7] (b6) -- node[dot, pos=0.125] (b7) {} ++(1,2);
			\draw[seq8] (b7) -- node[dot, pos=0.111] (b8) {} ++(-1,2);
		\end{scope}
		\node [above=10pt of b7] (bd) {$\vdots$};
		\node[point, seq9] [above=1pt of bd] (bw) {};
		\begin{scope}[every path/.style={thick}]
			\draw[seq10] (bw)  -- node[dot, pos=0.2] (bw1) {} ++(1,2);
			\draw[seq11] (bw1) -- node[dot, pos=0.333] (bw2) {} ++(-1,2);
		\end{scope}
	\end{tikzpicture}
	\caption{An example of growing a branchwise-real tree for $\omega+2$ steps. Spines beginning at the same vertical height are added at the same stage. Each stage is also represented using a different colour. At stage $\omega$, there are two new branches through the tree, the one on the right and the one on the left. We only extend the right, and we do so by adding a point at the limit, which is then further extended in subsequent stages.}
	\label{fig:growing tree}
\end{figure}
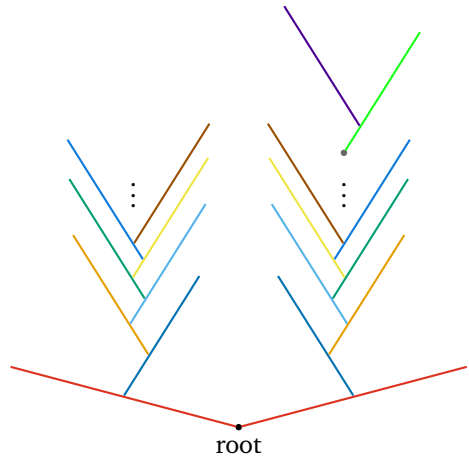

This construction may continue for $\omega_1$-many steps. Note however that since every branch must be order-isomorphic to a real interval, and there is no increasing $\omega_1$-sequence of reals, no branch through the final tree can contain a piece from every stage below $\omega_1$.

\begin{remark}
	A similar construction was used by Urysohn in \cite{UrysohnBeispielEN} (in German) to construct a metric space which is nowhere separable. Urysohn essentially followed the process outlined above, adding $\cont$-many open spines above each point added at the previous step, for $\omega$-many steps. Of course, Urysohn was constructing a topological space, not a partial order, but there is a clear correspondence between the two types of structure. This correspondence is elaborated in more detail in \cite{gradability-paper-published}. In \cite{berestovskii2019urysohn}, Urysohn's construction is defined (in English) and extended to all countable steps, by taking the Hausdorff completion at limit steps.
\end{remark}

Now, given any branchwise-real tree $X$ constructed in the fashion described above, we can naturally define a rank function $\rho \colon X \to \omega_1$, where $\rho(x)$ is the stage $\alpha$ at which $x$ was introduced. As \cref{lem:every to rank function} below demonstrates, any branchwise-real tree can be realised using the construction above, and thus admits a rank function. Note that rank functions are not unique in general. In fact, the same tree can have rank functions with a variety of ranges, as \cref{fig:alternative rank function} demonstrates.

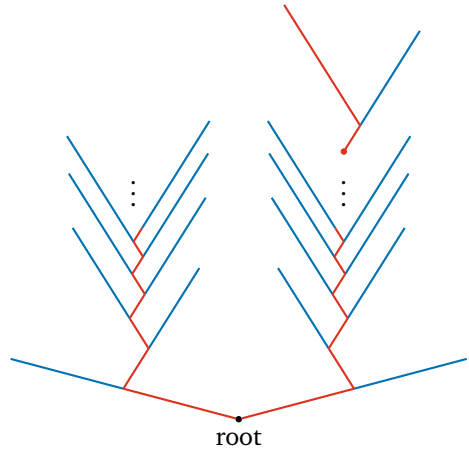
\begin{figure}[ht]
	\centering
	\begin{tikzpicture}[yscale=0.8]
		\node[point, seq0] at (0,0) [label=below:{root}] (r) {};

		\begin{scope}[every node/.style={shape=coordinate}]
			\path (r)  -- node[pos=0.500] (a1) {} ++(-3,1) node (a1t) {};
			\path (a1) -- node[pos=0.333] (a2) {} ++(1,2) node (a2t) {};
			\path (a2) -- node[pos=0.250] (a3) {} ++(-1,2) node (a3t) {};
			\path (a3) -- node[pos=0.200] (a4) {} ++(1,2) node (a4t) {};
			\path (a4) -- node[pos=0.167] (a5) {} ++(-1,2) node (a5t) {};
			\path (a5) -- node[pos=0.143] (a6) {} ++(1,2) node (a6t) {};
			\path (a6) -- node[pos=0.125] (a7) {} ++(-1,2) node (a7t) {};
			\path (a7) -- node[pos=0.111] (a8) {} ++(1,2) node (a8t) {};
		\end{scope}
		\begin{scope}[every path/.style={thick}]
			\draw[seq1] (r) -- (a1); \draw[seq2] (a1) -- (a1t);
			\draw[seq1] (a1) -- (a2); \draw[seq2] (a2) -- (a2t);
			\draw[seq1] (a2) -- (a3); \draw[seq2] (a3) -- (a3t);
			\draw[seq1] (a3) -- (a4); \draw[seq2] (a4) -- (a4t);
			\draw[seq1] (a4) -- (a5); \draw[seq2] (a5) -- (a5t);
			\draw[seq1] (a5) -- (a6); \draw[seq2] (a6) -- (a6t);
			\draw[seq1] (a6) -- (a7); \draw[seq2] (a7) -- (a7t);
			\draw[seq1] (a7) -- (a8); \draw[seq2] (a8) -- (a8t);
		\end{scope}
		\node [above=10pt of a7] (ad) {$\vdots$};

		\begin{scope}[every node/.style={shape=coordinate}]
			\path (r)  -- node[pos=0.500] (b1) {} ++(3,1) node (b1t) {};
			\path (b1) -- node[pos=0.333] (b2) {} ++(-1,2) node (b2t) {};
			\path (b2) -- node[pos=0.250] (b3) {} ++(1,2) node (b3t) {};
			\path (b3) -- node[pos=0.200] (b4) {} ++(-1,2) node (b4t) {};
			\path (b4) -- node[pos=0.167] (b5) {} ++(1,2) node (b5t) {};
			\path (b5) -- node[pos=0.143] (b6) {} ++(-1,2) node (b6t) {};
			\path (b6) -- node[pos=0.125] (b7) {} ++(1,2) node (b7t) {};
			\path (b7) -- node[pos=0.111] (b8) {} ++(-1,2) node (b8t) {};
		\end{scope}
		\begin{scope}[every path/.style={thick}]
			\draw[seq1] (r) -- (b1); \draw[seq2] (b1) -- (b1t);
			\draw[seq1] (b1) -- (b2); \draw[seq2] (b2) -- (b2t);
			\draw[seq1] (b2) -- (b3); \draw[seq2] (b3) -- (b3t);
			\draw[seq1] (b3) -- (b4); \draw[seq2] (b4) -- (b4t);
			\draw[seq1] (b4) -- (b5); \draw[seq2] (b5) -- (b5t);
			\draw[seq1] (b5) -- (b6); \draw[seq2] (b6) -- (b6t);
			\draw[seq1] (b6) -- (b7); \draw[seq2] (b7) -- (b7t);
			\draw[seq1] (b7) -- (b8); \draw[seq2] (b8) -- (b8t);
		\end{scope}
		\node [above=10pt of b7] (bd) {$\vdots$};
		\node[point, seq1] [above=1pt of bd] (bw) {};
		\begin{scope}[every node/.style={shape=coordinate}]
			\path (bw)  -- node[pos=0.2] (bw1) {} ++(1,2) node (bw1t) {};
			\path (bw1) -- node[pos=0.333] (bw2) {} ++(-1,2) node (bw2t) {};
		\end{scope}
		\begin{scope}[every path/.style={thick}]
			\draw[seq1] (bw) -- (bw1); \draw[seq2] (bw1) -- (bw1t);
			\draw[seq1] (bw1) -- (bw2t);
		\end{scope}
	\end{tikzpicture}
	\caption{An alternative construction of the tree in \cref{fig:growing tree}, which yields a different rank function. The root gets rank $0$. Both the branch which passes through every spine on the left and the branch which passes through every spine on the right have rank $1$. The rest of the tree has rank $2$. While the original construction produced $\omega+2$ ranks, this only produces $3$.}
	\label{fig:alternative rank function}
\end{figure}

\begin{definition}\label{def:rank function}
	Let $X$ be a tree order. A \emph{rank function} on $X$ is an order-preserving function $\rho \colon X \to \gamma$, where $\gamma$ is some ordinal, such that the following hold.
	\begin{enumerate}[label=(RF\arabic*), leftmargin=*, labelindent=3pt]
		\item\label{item:surjective; def:rank function}
		For every branch $B$ in $X$ the set $\{\rho(x) \mid x \in B\}$ is downwards-closed.
		\item\label{item:0 or limit; def:rank function}
		If $\alpha$ is $0$ or a limit, then $\rho^{-1}[\{\alpha\}]$ is an antichain.
		\item\label{item:successor; def:rank function}
		If $\alpha$ is a successor then $\rho^{-1}[\{\alpha\}]$ is the disjoint union of a family of incomparable rays.
	\end{enumerate}
\end{definition}

\begin{lemma}\label{lem:every to rank function}
	Every branchwise-real tree admits a rank function.
\end{lemma}

\begin{proof}
	Let $X$ be a tree order. By Zorn's Lemma, there is a maximal partial rank function $\rho$ whose domain is a downwards closed subset $Y \sse X$. Suppose for a contradiction that there is $x \in X \setminus Y$. Consider $\ds x \cap Y$. Note that $\ds x \cap Y$ is non-empty, since we can always give the root rank $0$. Furthermore, we can assign any limit of ranked points the rank which is the limit of the ranks of those points. Hence $\ds x \cap Y$ has a maximum element $y$. Let $C$ be the component above $y$ which contains $x$, and pick a maximal ray $R$ through $C$ which contains $x$. Then we can extend $\rho$ so that $x$ becomes ranked by giving every element of $R$ rank $\rho(y)+1$. \contradiction
\end{proof}

\begin{remark}
	\Cref{lem:every to rank function} shows that the construction elaborated above of iteratively adding spines is a completely general way of building branchwise-real trees. It was necessary to permit closed spines in addition to open spines in order to allow for trees with terminal nodes. All trees produced in this article, however, are without terminal nodes. Hence all constructions from now on will use open spines exclusively.
\end{remark}

Let us now establish some facts concerning rank functions on a branchwise-real tree $X$. Since every branch is order-isomorphic to a real interval, we can assume that the codomain of any rank function on $X$ is $\omega_1$.

\begin{definition}
	A rank function $\rho \colon X \to \omega_1$ is \emph{bounded} if the supremum of the set $\{\rho(x) \mid x \in X\}$ is less than $\omega_1$.
\end{definition}

\begin{lemma}\label{lem:countable height gradable}
	If $X$ has a bounded rank function, then it is continuously gradable.
\end{lemma}

To prove this lemma, we make use of the following set-theoretic result due to Cantor \cite{cantor-1895}, which in particular shows that every countable ordinal embeds into \Q.

\begin{lemma}\label{lem:countable ordinal in Q}
 	The set \Q\ of rationals is universal for countable linear orders: every countable linear order $X$ embeds into \Q.
\end{lemma}

\begin{proof}
	A slick proof makes use of the `forth' part of the classical `back-and-forth method'. Enumerate $X = \{x_n \mid n \in \omega\}$. By induction on this enumeration we then build up an embedding $r \colon X \to \Q$. Once we have $r \restr \{x_0, \ldots, x_{n-1}\}$, since \Q\ is a dense linear order without endpoints, we can find $r(x_{n})$ whose relative position with respect to $r(x_0), \ldots, r(x_{n-1})$ is the same that of $x_n$ with respect to $x_0, \ldots, x_{n-1}$.
\end{proof}

\begin{proof}[Proof of \cref{lem:countable height gradable}]
	Let $\rho \colon X \to \omega_1$ be a rank function with:
	\begin{equation*}
		\gamma \defeq \{\rho(x) \mid x \in X\} < \omega_1
	\end{equation*}
	Let $r \colon \gamma \to \Q$ be an embedding furnished by \cref{lem:countable ordinal in Q}. We construct an embedding $\ell \colon X \to \R$ by induction on the rank. If $\alpha$ is $0$ or a limit, map each element in $\rho^{-1}[\{\alpha\}]$ to $r(\alpha)$. Take a successor $\beta + 1$. Then $\rho^{-1}[\{\beta+1\}]$ is the disjoint union of a family of incomparable rays, each of which is order-isomorphic to a real interval. Take for each an isomorphism onto a sub-interval of $(r(\beta),r(\beta+1))$, and use this to extend $\ell$ there. After $\gamma$-many steps, we obtain an \R-grading $X \to \R$. Then by \cref{thm:R grading to continuous grading}, we find that $X$ is continuously gradable.
\end{proof}

\begin{remark}\label{rem:cont grad not implies bounded rank}
	The converse of \cref{lem:countable height gradable} does not hold. A good counterexample is $U_\kappa$, the universal continuously gradable branchwise real tree of degree $\kappa \geq 2$, defined as follows. This is essentially the construction of the universal \R-tree given in \cite{dyubina01}; see also \cite{nikiel1989topologies,mayer92}. Let $U_\kappa$ be the set of functions $r \colon [0,a) \to \kappa$, for each nonnegative real $a$, which are `piecewise constant to the right': for any $t \in [0,a)$ there is $\ep > 0$ such that $r \restr [t, t+\ep)$ is constant. Note that this condition ensures that the set of `value-change points' — those $t \in [0,a)$ for which for every $\ep > 0$ there is $s \in (t-\ep,t)$ such that $r(s) \neq r(t)$ — is well-ordered. 
	
	Then $U_\kappa$ becomes a branchwise-real tree under function extension. We can define a continuous grading $\ell$ on $U_\kappa$ by setting $\ell(r) = a$ when $r \colon [0,a) \to \kappa$. With this definition, any continuously graded branchwise-real tree of degree at most $\kappa$ embeds into $U_\kappa$ in such a way that the continuous gradings agree. 

	Let us see briefly why $U_\kappa$ has no bounded rank function. Suppose $\rho \colon X \to \gamma$ is a (surjective) rank function with $\gamma < \omega_1$. Assume that $\gamma$ is a limit ordinal, the successor case being similar. By \cref{lem:countable ordinal in Q} we can find an order embedding $s \colon \gamma + 1 \to [0, \infty)$ such that $s(0) = 0$. We can recursively define $r \colon [0, s(\gamma)) \to \kappa$ such that $\rho(r \restr [0, s(\alpha))) = \alpha$ for each $\alpha < \gamma$ as follows. To define $r$ on $[s(\alpha), s(\alpha+1))$, choose a ray above $r \restr [0, s(\alpha))$ of rank $\alpha + 1$. This ray is essentially a function $h_\alpha \colon [s(\alpha), \infty) \to \kappa$. Define:
	\begin{equation*}
		r \restr [s(\alpha), s(\alpha+1)) \defeq h_\alpha \restr [s(\alpha), s(\alpha+1))
	\end{equation*}
	At limits we take unions. But now $r$ cannot have a rank (it would get rank $\gamma$), which is a contradiction. \contradiction
\end{remark}

In \cref{sec:strongly uniform} below, we shall be considering the interaction between rank functions, connected components and order-automorphisms. The following basic facts will be useful.

\begin{definition}
	Let $\ab{X,\rho}$ be a ranked tree order. Take $x \in X$ and $C$ a component above $x$. The \emph{rank} of $C$ is $\rho(C) \defeq \min\{\rho(y) \mid y \in C\}$.
\end{definition}

Note that the rank of a connected component is always a successor.

\begin{lemma}\label{lem:component facts}
	Let $\ab{X,\rho}$ be a ranked tree order and take $x \in X$ non-terminal.
	\begin{enumerate}[label=(\arabic*)]
		\item\label{item:limit degrees; lem:component facts}
			When $\rho(x)$ is $0$ or a limit, every component above $x$ has rank $\rho(x)+1$.
		\item\label{item:successor degrees; lem:component facts}
			When $\rho(x)$ is a successor, there is one component above $x$ of rank $\rho(x)$ and the rest are of rank $\rho(x)+1$.
	\end{enumerate}
	Let $f \colon X \to X$ be an order-automorphism.
	\begin{enumerate}[resume*]
		\item\label{item:auto factors; lem:component facts}
			Then $f$ factors through a bijection of the components above $x$ onto the components above $f(x)$.
		\item\label{item:cointial interval; lem:component facts}
			If $f$ maps the component $C$ above $x$ onto the component $f(C)$ above $f(x)$, then there is a coinitial interval $I \sse C$ of constant rank $\rho(C)$ which maps onto a coinitial interval $f(I) \sse f(C)$ of constant rank $\rho(f(C))$.
	\end{enumerate}
\end{lemma}

\begin{proof}
	\begin{enumerate}[label=(\arabic*)]
		\item Let $C$ be any component above $x$, and pick $y \in C$ of minimal rank. Then the restriction of $\rho$ to $\ds y$ is surjective onto $\rho(y)+1$, from which we conclude that $\rho(C) = \rho(y) = \rho(x)+1$.

		\item Since $\rho(x)$ is a successor, $x$ lies on a ray all of whose elements have rank $\rho(x)$. Then as $\rho$ is order-preserving, the component above $x$ through which this ray passes must have rank $\rho(x)$. Let $C$ be any other component above $x$. Then $C$ cannot have rank $\rho(x)$, since $\rho^{-1}[\{\rho(x)\}]$ consists of the disjoint union of a family of incomparable rays. By the above argument, we have that $\rho(C) = \rho(x)+1$.

		\item Let $C$ be any component above $x$. If $y',z' \in f(C)$, then there is $w > x$ such that $w \leq f^{-1}(y'), f^{-1}(z')$. Hence we have $f(x) < f(w) \leq y',z'$, so that $y'$ and $z'$ lie in the same component above $f(x)$. By the same argument applied to $f^{-1}$, we see that $f$ maps components above $x$ onto components above $f(x)$, and vice versa.

		\item There is a ray $R$ through $C$ all of whose elements are of rank $\rho(C)$. Similarly, there is a ray $R'$ through $f(C)$ all of whose elements are of rank $\rho(f(C))$. Take $y' \in R'$ and $z' \in f(R)$. Then there is $w' > f(x)$ such that $w' \leq y', z'$. This $w'$ then lies on both rays. Since $f$ is an automorphism, this means that the coinitial interval $(x,f^{-1}(w')]$ is mapped onto the coinitial interval $(f(x),w']$.\qedhere
	\end{enumerate}
\end{proof}

%% file: parts/4-weakly-uniform.tex

\section{A weakly uniformly branching, rigid branch\-wise-real tree}
\label{sec:weakly uniform}

In this section, I construct our first rigid branchwise-real tree. It satisfies the following uniformity condition, the stronger version of which will be considered in the next section.

\begin{definition}
	A tree order is \emph{weakly uniformly $\kappa$-branching} if every branching node has degree $\kappa$.
\end{definition}

\begin{theorem}\label{thm:weakly-uniform rigid}
	Let $2 \leq \kappa \leq \cont$. There is a weakly uniformly $\kappa$-branching, rigid branchwise-real tree which is continuously gradable.
\end{theorem}

\begin{proof}
	We first construct a collection of dense, mutually non-isomorphic sets of positive reals $\{S_A \sse (0,\infty) \mid A \sse \omega\}$, such that $S_A$ `encodes' $A$.\footnote{In fact, we only need that the $S_A$'s are mutually non-isomorphic as subsets of $(0, \infty)$, but the construction yields sets which are mutually non-isomorphic as stand-alone linear orders.} The set $S_A$ consists of the rationals in $(0,\infty)$ plus a descending sequence of real intervals with limit $0$ whose endpoints are irrational, such that the $n$th interval is open when $n \notin A$ and half-open when $n \in A$. Formally, fix some descending sequence $(x_n)$ of irrationals with limit $0$. Then for any $A \sse \omega$ we define:
	\begin{equation*}
		S_A \defeq (\Q \cap (0,\infty)) \cup \bigcup_{n \in \omega}(x_{2n+1},x_{2n}) \cup \{x_{2n} \mid n \in A\}
	\end{equation*}
	See \cref{fig:example S_A; proof:weakly-uniform rigid} for a representation of an example $S_A$.

	\begin{figure}[ht]
		\centering
		\begin{tikzpicture}[xscale=1.3]
			\def\subset{{1,0,0,1,0,0,1,1,1}}
			\pgfmathsetmacro{\subsetmax}{8}
			\draw[rounded corners=10] 
				(8,0.5) -- (0,0.5) -- (0,-0.5) -- (8,-0.5);
			\begin{scope}
				\clip [rounded corners=10] (8,0.5) -- (0,0.5) -- (0,-0.5) -- (8,-0.5);
				\fill[pattern=dots] (1,0.5) rectangle (8,-0.5);
				\foreach \n in {0,...,\subsetmax}
				{
					\pgfmathsetmacro{\xmin}{{9*(1-atan((2*\n+1)/7)/90)-1}}
					\pgfmathsetmacro{\xmax}{{9*(1-atan((2*\n+2)/7)/90)-1}}
					\pgfmathsetmacro{\xmed}{{(\xmin+\xmax)/2}}
					\pgfmathsetmacro\subsetvalue{\subset[\n]}
					\ifthenelse{\equal{\subsetvalue}{1}}
					{
						\path[fill=white]
							(\xmin,0.5) 
								{[rounded corners=5] -- (\xmax,0.5)
								-- (\xmax,-0.5)}
								-- (\xmin,-0.5)
								-- cycle;
						\draw[fill=black]
							(\xmin,0.5) 
								{[rounded corners=5] -- (\xmax,0.5)
								-- (\xmax,-0.5)}
								-- (\xmin,-0.5)
								-- cycle;
					}
					{
						\path[fill=white, rounded corners=5]
							(\xmin,0.5) 
								-- (\xmax,0.5)
								-- (\xmax,-0.5)
								-- (\xmin,-0.5)
								-- cycle;
						\draw[fill=black, rounded corners=5]
							(\xmin,0.5) 
								-- (\xmax,0.5)
								-- (\xmax,-0.5)
								-- (\xmin,-0.5)
								-- cycle;
					}
					\coordinate (i\n) at (\xmin,-0.5);
				}
				\node at (0.6,0) {$\cdots$};
			\end{scope}
			\begin{scope}[yshift=-70, xshift=10]
				\node at (0.5,0) {$A=$};
				\node at (1,0) {$\bigg\{$};
				\node at (1.75,-0.1) {$\ldots,$};
				\node (a8) at (2.5,0) {$8,$};
				\node (a7) at (3.25,0) {$7,$};
				\node (a6) at (4.0,0) {$6,$};
				\node (a3) at (4.75,0) {$3,$};
				\node (a0) at (5.5,0) {$0$};
				\node at (6,0) {$\bigg\}$};
				\foreach \n in {0,3,6,7,8}
				{
					\draw[-latex] (a\n) to[out=90, in=270] ($(i\n)-(0,0.02)$);
				}
			\end{scope}
		\end{tikzpicture}
		\caption{A representation of an example $S_A$ set. Solid regions contain every real, while dotted regions only contain the rationals. A curved boundary on a solid region signifies that it does not include that endpoint. The set $A \sse \omega$ is written in descending order, so as to coincide with the ordering of the real intervals. Arrows relate the elements of $A$ with the corresponding endpoints of real intervals which code for them in $S_A$.}
		\label{fig:example S_A; proof:weakly-uniform rigid}
	\end{figure}
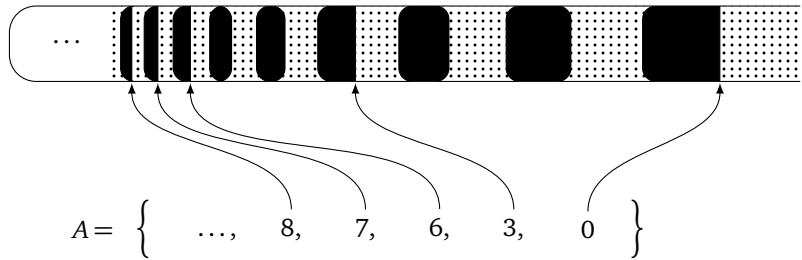
	
	Suppose that $p \colon S_A \to S_B$ is an order-isomorphism for $A \neq B$. Then $p$ must map the real intervals onto the real intervals. Moreover, the rightmost real interval in $S_A$ must map to the rightmost real interval in $S_B$, and so on. But $A$ and $B$ disagree on some natural number, say $n \in A \setminus B$ (without loss of generality). Then the $n$th real interval from the right in $S_A$ is half-open, and $p$ sends it to the $n$th real interval in $S_B$, which is open. Since we ensured that the endpoints of the real intervals were irrational, this means that $p$ doesn't preserve the fact that the $n$th real interval in $S_A$ has a supremum. \contradiction

	Our rigid branchwise-real tree $X$ is now constructed in $\omega$-many stages using open spines, laying along each new spine a different set $S_A$, and using this to determine which points to use as the bases of new spines. Note that we need to make sure that we have fresh $S_A$ sets available at each stage of the construction: we don't want to run out. For this, we can partition $\P(\omega)$ into $\omega$-many batches of size $\cont$, so that the $n$th batch is reserved for the $n$th stage.

	We construct $X$ in stages $X_n$ for $n \in \omega$. Start with $X_0$ the singleton tree, and form $X_1$ by adding $\kappa$-many new open spines above the root. Now assume that $X_n$ is constructed for $n \geq 1$. To each spine $U$ added at stage $n$, associate a different, unused subset $A \sse \omega$. Fix an isomorphism $p_U \colon U \to (0,\infty)$ and above each element in $p_U^{-1}(S_A)$ add $(\kappa-1)$-many new open spines. Finally, let $X \defeq \bigcup_{n \in \omega}X_n$. Define a rank function $\rho \colon X \to \omega$, so that for each $n \in \omega$ we have $\rho^{-1}[\{n+1\}] = X_{n+1} \setminus X_n$.

	Note that at each stage of the construction, we made sure that any branching points got degree $\kappa$. Furthermore, $\rho$ is a bounded rank function on $X$, hence by \cref{lem:countable height gradable} we have that $X$ is continuously gradable.

	Let us see that $X$ is rigid. Suppose for a contradiction that $f \colon X \to X$ is a non-trivial order-automorphism. Then there is $x \in X$ such that $f(x) \neq x$. There must then be $y > x$ such that $f([x,y])$ and $[x,y]$ are disjoint. Since each $S_A$ is dense in $(0,\infty)$, there must be $z \in [x,y]$ which is branching. Choose any spine $U$ above $z$ added during the construction. Then $f(U)$ is a ray through $\us{f(z)}$ and $f \restr U \colon U \to f(U)$ is an isomorphism such that $w \in U$ is branching if and only if $f(w) \in f(U)$ is branching. 

	Let us now consider the pattern of branching nodes along $U$ and $f(U)$. Following $U$, this looks like some $S_A$: it consists in an alternating sequence of real intervals and rational intervals, which converge at the base. Hence we must have the same pattern along $f(U)$. Now, there three possibilities for the way in which $f(U)$ fits into our construction.
	\begin{enumerate*}[label=(\roman*)]
		\item\label{item:single; possibilities; proof:weakly-uniform rigid}
			$f(U)$ is a single spine added above $f(z)$.
		\item\label{item:mixed; possibilities; proof:weakly-uniform rigid}
			$f(U)$ consists of an initial segment of a spine added above $f(z)$, followed by an initial segment of a spine added at the next stage (followed potentially by further initial segments of spines added at later stages).
		\item\label{item:part-way; possibilities; proof:weakly-uniform rigid}
			$f(U)$ starts part-way up a spine added at an earlier stage.
	\end{enumerate*}
	See \cref{fig:F(U) patterns; proof:weakly-uniform rigid} for a representation of the three types of patterns of branching points which can occur along $f(U)$. Considered as subsets of $(0, \infty)$, no two of three possible patterns of branching nodes can be order-isomorphic. Since the pattern of branching nodes along $U$ looks like \ref{item:single; possibilities; proof:weakly-uniform rigid}, the pattern along $f(U)$ must also have this form.

	\begin{figure}[ht]
		\centering
		\begin{tikzpicture}[xscale=1.3, yscale=0.95]
			\begin{scope}
				\node at (-0.5,0) {\ref{item:single; possibilities; proof:weakly-uniform rigid}};
				\pgfmathsetmacro{\intervalmax}{25}
				\draw[rounded corners=10] 
					(8,0.5) -- (0,0.5) -- (0,-0.5) -- (8,-0.5);
				\begin{scope}
					\clip [rounded corners=10] (8,0.5) -- (0,0.5) -- (0,-0.5) -- (8,-0.5);
					\fill[pattern=dots] (0,0.5) rectangle (8,-0.5);
					\foreach \n in {0,...,\intervalmax}
					{
						\pgfmathsetmacro{\xmin}{{9*(1-atan((2*\n+1)/7)/90)-0.8}}
						\pgfmathsetmacro{\xmax}{{9*(1-atan((2*\n+2)/7)/90)-0.8}}
						\pgfmathsetmacro{\rounding}{{10*(1-atan((2*\n+2)/7)/90)}}
						\path[fill=white]
							(\xmin,0.5) 
								{[rounded corners=\rounding] -- (\xmax,0.5)
								-- (\xmax,-0.5)}
								-- (\xmin,-0.5)
								-- cycle;
						\draw[fill=black]
							(\xmin,0.5) 
								{[rounded corners=\rounding] -- (\xmax,0.5)
								-- (\xmax,-0.5)}
								-- (\xmin,-0.5)
								-- cycle;
					}
				\end{scope}
			\end{scope}
			\begin{scope}[yshift=-50]
				\node at (-0.5,0) {\ref{item:mixed; possibilities; proof:weakly-uniform rigid}};
				\pgfmathsetmacro{\intervalamax}{25}
				\pgfmathsetmacro{\intervalamin}{4}
				\pgfmathsetmacro{\intervalbmax}{25}
				\pgfmathsetmacro{\intervalbmin}{2}
				\draw[rounded corners=10] 
					(8,0.5) -- (0,0.5) -- (0,-0.5) -- (8,-0.5);
				\begin{scope}
					\clip [rounded corners=10] (8,0.5) -- (0,0.5) -- (0,-0.5) -- (8,-0.5);
					\fill[pattern=dots] (0,0.5) rectangle (8,-0.5);
					\foreach \n in {\intervalamin,...,\intervalamax}
					{
						\pgfmathsetmacro{\xmin}{{9*(1-atan((2*\n+1)/7)/90)-0.8}}
						\pgfmathsetmacro{\xmax}{{9*(1-atan((2*\n+2)/7)/90)-0.8}}
						\pgfmathsetmacro{\rounding}{{10*(1-atan((2*\n+2)/7)/90)}}
						\path[fill=white]
							(\xmin,0.5) 
								{[rounded corners=\rounding] -- (\xmax,0.5)
								-- (\xmax,-0.5)}
								-- (\xmin,-0.5)
								-- cycle;
						\draw[fill=black]
							(\xmin,0.5) 
								{[rounded corners=\rounding] -- (\xmax,0.5)
								-- (\xmax,-0.5)}
								-- (\xmin,-0.5)
								-- cycle;
					}
					\foreach \n in {\intervalbmin,...,\intervalbmax}
					{
						\pgfmathsetmacro{\xmin}{{9*(1-atan((2*\n+1)/7)/90)-1+3.4}}
						\pgfmathsetmacro{\xmax}{{9*(1-atan((2*\n+2)/7)/90)-1+3.4}}
						\pgfmathsetmacro{\rounding}{{10*(1-atan((2*\n+2)/7)/90)}}
						\path[fill=white]
							(\xmin,0.5) 
								{[rounded corners=\rounding] -- (\xmax,0.5)
								-- (\xmax,-0.5)}
								-- (\xmin,-0.5)
								-- cycle;
						\draw[fill=black]
							(\xmin,0.5) 
								{[rounded corners=\rounding] -- (\xmax,0.5)
								-- (\xmax,-0.5)}
								-- (\xmin,-0.5)
								-- cycle;
					}
				\end{scope}
			\end{scope}
			\begin{scope}[yshift=-100]
				\node at (-0.5,0) {\ref{item:part-way; possibilities; proof:weakly-uniform rigid}};
				\pgfmathsetmacro{\intervalmax}{6}
				\draw[rounded corners=10] 
					(8,0.5) -- (0,0.5) -- (0,-0.5) -- (8,-0.5);
				\begin{scope}
					\clip [rounded corners=10] (8,0.5) -- (0,0.5) -- (0,-0.5) -- (8,-0.5);
					\fill[pattern=dots] (0,0.5) rectangle (8,-0.5);
					\foreach \n in {0,...,\intervalmax}
					{
						\pgfmathsetmacro{\xmin}{{9*(1-atan((2*\n+1)/7)/90)-1-2.1}}
						\pgfmathsetmacro{\xmax}{{9*(1-atan((2*\n+2)/7)/90)-1-2.1}}
						\pgfmathsetmacro{\rounding}{{10*(1-atan((2*\n+2)/7)/90)}}
						\path[fill=white]
							(\xmin,0.5) 
								{[rounded corners=\rounding] -- (\xmax,0.5)
								-- (\xmax,-0.5)}
								-- (\xmin,-0.5)
								-- cycle;
						\draw[fill=black]
							(\xmin,0.5) 
								{[rounded corners=\rounding] -- (\xmax,0.5)
								-- (\xmax,-0.5)}
								-- (\xmin,-0.5)
								-- cycle;
					}
				\end{scope}
			\end{scope}
		\end{tikzpicture}
		\caption{A comparison three possible types of pattern which can occur along $f(U)$. Solid regions contain every real, while dotted regions only contain the rationals.}
		\label{fig:F(U) patterns; proof:weakly-uniform rigid}
	\end{figure}
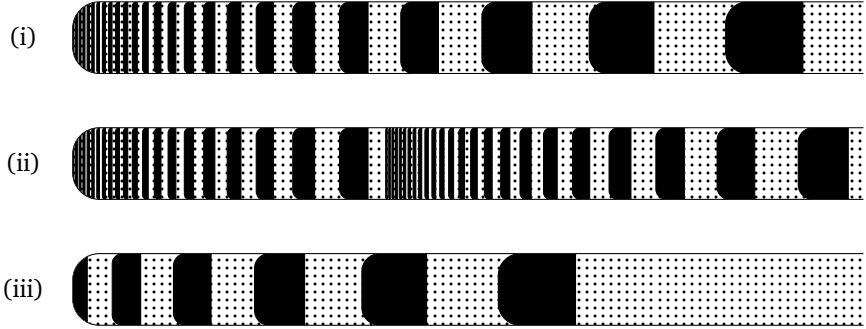 

	Therefore, there are $A,B \sse \omega$ distinct, such that the pattern of branching nodes along $U$ looks like $S_A$, and the pattern of branching nodes along $f(U)$ looks like $S_B$. In other words, we have order-isomorphisms $p_U \colon U \to (0,\infty)$ and $p_{f(U)} \colon f(U) \to (0,\infty)$ such that $p_{f(U)} \circ f \circ p_U^{-1} \colon (0,\infty) \to (0,\infty)$ is an order-automorphism mapping $S_A$ onto $S_B$. This is our contradiction. \contradiction
\end{proof}

Let us now consider the relationship between the construction just given and a forcing notion. For ease of exposition, we only deal with the case $\kappa = \aleph_0$. We force a weakly uniformly $\aleph_0$-branching, continuously gradable, rigid branchwise-real tree by considering countable approximations to it. Such objects might be called \emph{sub-branchwise-real trees}: we replace \ref{item:interval; def:brot} with the requirement that each branch be embeddable into \R. In addition, we keep track of a rank function from the approximations into $\omega$, together with a set of `non-extension promises'. The latter is a subset of the sub-branchwise-real tree consisting of degree-$1$ nodes which guarantees that these nodes never become branching. The forcing conditions are thus triples $\ab{X, \rho, S}$ consisting of: 
\begin{enumerate*}[label=(\alph*)]
	\item a countable sub-branchwise-real tree $X$ (which for concreteness we can take to be a subset of $\omega_1$),
	\item a rank function $\rho \colon X \to \omega$, and
	\item a non-extension promise set $S \sse X$ of degree-$1$ nodes.
\end{enumerate*}
For one condition to be stronger than another, we require the sub-branchwise-real tree, rank function and promise set of the former extend those of the latter, and further that we do not add new connected components above any already-branching point, so that in particular the degrees of branching points do not decrease. To ensure this latter condition, we require that for every element $x$ of the weaker approximation, every component above $x$ with respect to the stronger approximation contain at most one component above $x$ with respect to the weaker. Note that the fact that stronger conditions must extend the promise sets and that every element of the promise set must have degree $1$ ensures that we keep our promises. Let $\bb P$ be the forcing notion just described. 

\begin{theorem}
	The forcing notion $\bb P$ is countably closed, and any $\bb P$-generic set determines an $\omega$-ranked, weakly uniformly $\aleph_0$-branching, rigid branchwise-real tree.
\end{theorem}

\begin{proof}[Proof sketch]
	That $\bb P$ is countably closed follows immediately from the definition. Denote by $\ol X$ the union of the approximations found in $G$. It is not hard to see by genericity that $\ol X$ is a weakly uniformly $\aleph_0$-branching branchwise-real tree, and that the union of the partial rank functions yields a total rank function $\ol X \to \omega$.

	To see that $\ol X$ is rigid, take $\dot f \colon \ol X \to \ol X$ a non-trivial automorphism. By a countable-closure argument, we can find a condition $\ab{X, \rho, S}$ forcing
	\begin{enumerate*}[label=(\roman*)]
		\item that $\dot f$ is a order-preserving function,
		\item that $X \sse \dom(\dot f)$ and that $\dot f$ has a decided value on every element of $X$ and
		\item that there is $x \in X$ such that $\dot f(x)$ is incomparable with $x$.
	\end{enumerate*}
	Moreover we can assume, using another countable-closure argument, that there is $z \geq x$ in $X$ such that both $z$ and $\dot f(z)$ are branching, and furthermore that there is $w > z$ such that $[z, w]^X$ is order-dense as subset of $[z, w]^{\ol X}$. Since $[z, w]^X$ is countable, it has continuum-many `holes' to fill. Any element which fills a hole has its image under $\dot f$ fixed, given then density of $[z, w]^X$ in $[z, w]^{\ol X}$. Finally, it is dense (in the $\bb P$-forcing sense) that some hole in $[z, w]^X$ is filled with a branching node while the corresponding hole in $[\dot f(z), \dot f(w)]^X$ is filled with a non-branching node which is promised to remain non-branching. Since $G$ is generic, this means that $\dot f$ cannot be an automorphism of $\ol X$, which is a contradiction. \contradiction
\end{proof}

%% file: parts/5-strongly-uniform.tex

\section{A uniformly branching, rigid branchwise-real tree}

\label{sec:strongly uniform}

In this section, I strengthen the result of the previous, by showing that for every $\kappa$ with $2 \leq \kappa \leq \cont^+$ there is a rigid branchwise-real tree in which every point is branching of the same degree $\kappa$.

\begin{definition}
	A tree order is \emph{uniformly $\kappa$-branching} if every node has degree $\kappa$.
\end{definition}

Moreover, we can ask that such a tree be either continuously gradable (\cref{thm:strongly-uniform rigid}) or not (\cref{thm:strongly-uniform rigid non gradable}).

\begin{theorem}\label{thm:strongly-uniform rigid}
	Let $2 \leq \kappa \leq \cont^+$. There is a uniformly $\kappa$-branching, rigid branchwise-real tree which is continuously gradable.
\end{theorem}

As in the proof of \cref{thm:weakly-uniform rigid}, our uniformly branching trees will be grown iteratively upwards. But notice that, since we require every point to have the same degree $\kappa$, at successor stages there are no choices to make: we must add $(\kappa-1)$-many new spines above every point added at the previous stage. Were we to stop this process after $\omega$-many steps, we would end up with a `minimal' uniformly branching tree, which is not only non-rigid, but moreover homogeneous, as the following result shows.

\begin{definition}
	Let $\kappa$ be a cardinal. The \emph{minimal uniformly $\kappa$-branching branchwise-real tree}, $M_\kappa$, is the branchwise-real tree built in $\omega$-many steps, starting with the root, and at successors stages adding $(\kappa-1)$-many new open spines above every point added at the previous stage.
\end{definition}

\begin{remark}
	The tree $M_\cont$ is the (underlying partial order of) the nowhere separable metric space constructed by Urysohn in \cite{UrysohnBeispielEN}.
\end{remark}

\begin{proposition}\label{prop:M kappa props}\
	\begin{enumerate}[label=(\arabic*)]
		\item $M_\kappa$ is minimal: for every uniformly $\kappa$-branching branchwise-real tree $X$ there is an order-preserving embedding $f \colon M_\kappa \to X$. Moreover, we can assume that $f$ is continuous: for any $x < y$ in $M_\kappa$ the restriction $f \restr [x,y] \colon [x,y] \to [f(x), f(y)]$ is surjective.
		\item $M_\kappa$ is homogeneous: if $x, y \in M_\kappa$ are not the root, there is an order-automorphism $f \colon M_\kappa \to M_\kappa$ such that $f(x) = y$.
	\end{enumerate}
\end{proposition}

\begin{proof}
	First note the construction of $M_\kappa$ yields a rank function $\rho \colon M_\kappa \to \omega$.
	\begin{enumerate}[label=(\arabic*)]
		\item We can build an order-preserving embedding $M_\kappa \to X$ using a Zorn's Lemma style argument, essentially following the proof of \cref{lem:every to rank function}.
		\item First take any $z \in M_\kappa$ such that $n = \rho(z)>1$. We define an automorphism $r_z \colon M_\kappa \to M_\kappa$ which steps down $z$'s rank: $\rho(r_z(z)) = n-1$. Let $w$ be the greatest element of $\ds z$ of rank $n-1$. Then $z$ lies on a spine added above $w$ at stage $n$. Let $C$ be the component above $w$ which contains $z$ (which component is of rank $n$). Let $D$ be the component above $w$ which contains the rest of the rank-$(n-1)$ spine on which $w$ lies (which is of rank $n-1$). See \cref{fig:twist; proof:M kappa props} for a picture of the situation. It is easy to define an isomorphism $s \colon C \to D$ such that the rank-$n$ spine on which $z$ lies is mapped to the rest of the rank-$(n-1)$ spine above $w$. Letting $r_z$ be the automorphism which swaps $C$ and $D$ via $s$, we see that $\rho(r_z(z)) = n-1$.

		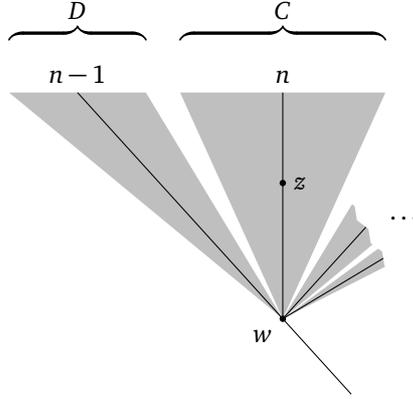
\begin{figure}[ht]
			\centering
			\begin{tikzpicture}[xscale=0.9]
				\fill[lightgray] (-4,3) -- (0,0) -- (-2,3);
				\draw (1,-1) -- (-3,3) node [above] {$n - 1$};
				\fill[lightgray] (-1.5,3) -- (0,0) -- (1.5,3);
				\draw (0,0) -- node[pos=0.6] (z) {} (0,3) node [above] {$n$} node (t1) {};
				\node[point] at (z) [label=right:$z$] {};
	
				\begin{scope}
					\clip[decoration={zigzag, amplitude=1pt, segment length=10pt}, decorate] (0,-1) -- (2,0) -- (0,3) -- (-1,0) -- (0,-1);
					\fill[lightgray] (2,3) -- (0,0) -- (4,3);
					\draw (0,0) -- (3,3);
					\fill[lightgray] (4.5,3) -- (0,0) -- (6.5,3);
					\draw (0,0) -- (5.5,3);
				\end{scope}
				\node at (1.8,1.3) {$\cdots$};
	
				\draw[brace] (-4,3.7) -- node[above=4pt] {$D$} (-2,3.7);
				\draw[brace] (-1.5,3.7) -- node[above=4pt] {$C$} (1.5,3.7);
	
				\node[point] at (0,0) [label=below left:$w$] {};
			\end{tikzpicture}
			\caption{The situation when reducing $z$'s rank. The point $z$ lies on a rank-$n$ spine added above $w$, in the component $C$ above $w$. The rank-$(n-1)$ spine on which $w$ lies continues into the component $D$ above $w$. The automorphism $r_z$ swaps the components $C$ and $D$.}
			\label{fig:twist; proof:M kappa props}
		\end{figure}

		Now take $x, y \in M_\kappa$ non-root. By repeating the above procedure, we can assume that both have rank $1$. Moreover, by performing another `twist', sending the component above the root containing $y$ to that containing $x$, we may assume that $x$ and $y$ are comparable, lying on the same spine; say $x \leq y$. Finally, it is not hard to define an automorphism of $M_\kappa$ whose restriction to the rank-$1$ spine containing $x$ and $y$ looks like an automorphism of $(0,\infty)$, and which maps $x$ to $y$.\qedhere
	\end{enumerate}
\end{proof}

Turning back to the construction of rigid trees, in the proof of \cref{thm:strongly-uniform rigid}, as well as \cref{thm:strongly-uniform rigid non gradable} below, the notion of a `colouring' of the positive real numbers $(0, \infty)$ plays an auxiliary role. A colouring of the positive real numbers is simply a function with domain $(0, \infty)$; we think of elements of the range as `colours'. 

\begin{definition}
	Let $X$ and $S$ be sets. An \emph{$S$-colouring} of $X$ is a function $X \to S$.
\end{definition}

In the proof of \cref{thm:weakly-uniform rigid}, the first step was essentially to construct a family of black-white colourings of $(0, \infty)$ with suitable properties. To construct a uniformly branching, rigid branchwise-real tree, we extend this idea. This time we look for a family of $2$-colourings (and later $\omega$-colourings) which is `sufficiently generic'. This means that for any pair of distinct colourings, any order-automorphism of $(0, \infty)$ and any pair of colours, we can densely often find a point coloured with the first colour under the first colouring whose image under the automorphism is coloured with the second colour under the second colouring. This is the following result.

\begin{lemma}[Generic Colouring Lemma]\label{lem:generic family colourings}
	Let $\lam \leq \cont$ be a cardinal. There is a family $\A$ with size $\cont^+$ of $\lam$-colourings $(0,\infty) \to \lam$ of the positive real numbers such that the following holds. For any $c,d \in \A$ distinct, for any order-automorphism $p \colon (0, \infty) \to (0, \infty)$ and for any $\alpha,\beta \in \lam$ there is a dense set of $x \in (0, \infty)$ such that $c(x) = \alpha$ and $d(p(x)) = \beta$.
\end{lemma}

To prove this we need the following basic result concerning automorphisms of the positive real numbers.

\begin{lemma}\label{lem:number of order-isomorphisms of R}
	There are exactly $\cont$-many order-automorphisms $(0, \infty) \to (0, \infty)$.
\end{lemma}

\begin{proof}
	Every real number is the limit of rationals below it, so any order-automorphism $(0, \infty) \to (0, \infty)$ is determined by its values on $(0, \infty) \cap \Q$, of which there are $\cont$-many possibilities.
\end{proof}

\begin{proof}[Proof of \cref{lem:generic family colourings}]
	By Zorn's Lemma, we can take a maximal family $\A$ of $\lam$-colourings, satisfying both the property in the statement, as well as that for every $c \in \A$:
	\begin{equation}
		\parbox{\textwidth-3cm}{for every $\alpha \in \lam$ and for every $b > a > 0$ there are continuum-many points $x \in (a, b)$ such that $c(x)=\alpha$} \tag{P}\label{state:contiuum; proof:generic family colourings}
	\end{equation}
	Let us first see that $\A$ is non-empty. For convenience, I will construct a $(0,1)$-colouring satisfying \eqref{state:contiuum; proof:generic family colourings}; this can then easily be adapted to a $(0, \infty)$-colouring. It suffices to take $\lam=\cont$. Consider the elements of $(0,1)$ as binary $\omega$-sequences representing binary expansions (where, to avoid ambiguity, we make some canonical choice for the binary representation of $x \in (0,1)$, in the case that two such representations are possible). Define an equivalence relation on these, where two sequences are related if one is a tail of the other. Each equivalence class is countable, hence there are continuum-many. Moreover, each equivalence class is dense in $(0,1)$. Group the classes into continuum-many batches of size $\cont$, and colour all of each batch with a different element of $\cont$. The resulting colouring satisfies \eqref{state:contiuum; proof:generic family colourings}.

	Now suppose for a contradiction that $\abs\A < \cont^+$. We will extend $\A$ by diagonalising against the previous colourings, and against every automorphism. Let $\Aut(0, \infty)$ be the set of order-automorphisms $(0, \infty) \to (0, \infty)$. By \cref{lem:number of order-isomorphisms of R}, the set $\Aut(0, \infty) \times \A \times \lam \times \lam$ has size $\cont$. Enumerate:
	\begin{equation*}
		\Aut(0, \infty) \times \A \times \lam \times \lam = \{(p_\theta,c_\theta, \alpha_\theta, \beta_\theta) \mid \theta < \cont\}
	\end{equation*}
	We build a new colouring $d \colon (0, \infty) \to \lam$ recursively in stages $d_\theta$ for ${\theta< \cont}$, each of which is a partial colouring with domain of size less than $\cont$. Assume we have constructed $d_\mu$ for $\mu < \theta$. First let $d_\theta' \defeq \bigcup_{\mu<\theta} d_\mu$. Now, since $c_\theta$ satisfies \eqref{state:contiuum; proof:generic family colourings}, and since $d_\theta'$ has domain of size less than $\cont$, we can find a countable dense set of $X \sse (0, \infty)$ such that for every $x \in X$ we have $c_\theta(x) = \alpha_\theta$ while $d_\theta'(p_\theta(x))$ is undefined. Extend $d_\theta'$ to $d_\theta$ by letting $d_\theta(p_\theta(x)) \defeq \beta_\theta$, for every $x \in X$. Finally, let $d$ be the union $\bigcup_{\theta < \cont} d_\theta$, filling in any points which remain uncoloured arbitrarily. But now, $\A \cup \{d\}$ is a larger family, contradicting the maximality of $\A$. \contradiction
\end{proof}

With this lemma established, we can now construct our uniformly $\kappa$-branching, rigid branchwise-real tree which is continuously gradable.

\begin{proof}[Proof of \cref{thm:strongly-uniform rigid}]
	Let $\A$ be a $\cont^+$-sized family of $2$-colourings of $(0,\infty)$ as per the Generic Colouring Lemma \ref{lem:generic family colourings}. This time, we construct $X$ in stages $X_\alpha$ for $\alpha < \omega^2$. When we add a new spine we lay a new colouring from $\A$ along it, and thus we also build colourings $c_\alpha \colon X_\alpha \to 2$. At limit stages, we use these colourings to decide which branches to extend. We also keep track of the rank $\rho$ of elements of $X$, so that the points added at stage $\alpha$ get rank $\alpha$. As before, we need to make sure we never run out of colourings, so we can partition $\A$ into $\omega^2$-many batches of size $\cont^+$.

	Start with $X_0$ the singleton, coloured $0$. Assume that $X_\alpha$ and $c_\alpha$ are constructed, where $\alpha$ is $0$ or a limit. To make $X_{\alpha+1}$, above all of the points of rank $\alpha$, we add $\kappa$-many new open spines. For each spine added, we pick a different, unused colouring $c \in \A$, and use it to define $c_{\alpha+1}$ on that spine. Now assume that $X_\alpha$ and $c_\alpha$ are constructed, where $\alpha$ is a successor. To make $X_{\alpha+1}$, above each point of rank $\alpha$, we add $(\kappa-1)$-many new open spines, again colouring them with new colourings from $\A$.

	So take $\alpha < \omega^2$ a limit, and assume we have constructed $X_\beta$ and $c_\beta$ for $\beta < \alpha$. First let $X_\alpha' \defeq \bigcup_{\beta < \alpha} X_\alpha$ and $c_\alpha' \defeq \bigcup_{\beta < \alpha} c_\alpha$. Then $X_\alpha'$ admits a number of new branches which appear in the limit: in other words, branches through $X_\alpha'$ containing elements of rank unbounded in $\alpha$. It is these branches which we will decide to extend or not. Let $B$ be any such branch. Then for every $\beta < \alpha$, there is a maximum element $x_\beta \in B$ of rank $\beta$. Consider the sequence $(c_\alpha'(x_0), c_\alpha'(x_1), \ldots)$ of the colours of such points. We will extend $B$ if and only if this sequence has a tail consisting of $1$'s. To extend $B$, add a new point to $X_\alpha'$ lying directly above $B$, and colour it $0$. Once we have carried out this procedure for every new branch through $X_\alpha'$, we arrive at our new pair $\ab{X_\alpha, c_\alpha}$.

	Finally, let $X \defeq \bigcup_{\alpha < \omega^2} X_\alpha$ and $c \defeq \bigcup_{\alpha < \omega^2} c_\alpha$. Note that at each stage we made sure that every point is branching of degree $\kappa$, so $X$ is uniformly $\kappa$-branching. Moreover, $\rho \colon X \to \omega^2$ is a bounded rank function, so $X$ is continuously gradable by \cref{lem:countable height gradable}.

	Let us see that $X$ is rigid. Suppose for a contradiction that $f \colon X \to X$ is a non-trivial automorphism. Then there is $x_0 \in X$ such that $f(x_0) \neq x_0$. When $\kappa \geq 3$, the proof is simpler, so let's deal with that case first. See \cref{fig:kappa geq 3; proof:strongly-uniform rigid} for a representation of the situation. Consider the components above above $x_0$ and above $f(x_0)$. By \cref{lem:component facts}, $f$ factors through a bijection of the components above $x_0$ to those above $f(x_0)$. Moreover, all but at most one component above $x_0$ is of rank $\rho(x_0)+1$, and similarly for $f(x_0)$. Since $\deg(x_0) = \deg(f(x_0)) = \kappa \geq 3$ there is a component $C$ above $x_0$ of rank $\rho(x_0)+1$ which maps to a component $f(C)$ above $f(x_0)$ of rank $\rho(f(x_0))+1$. Furthermore, using \cref{lem:component facts} again, there is an coinitial interval $I \sse C$ of constant rank $\rho(x_0)+1$ which maps to a coinitial interval $f(I) \sse f(C)$ of constant rank $\rho(f(x_0))+1$. Then $I$ is an initial segment of a spine $U$ added at stage $\rho(x_0)+1$, while $f(I)$ is an initial segment of a spine $V$ added at stage $\rho(f(x_0))+1$.

	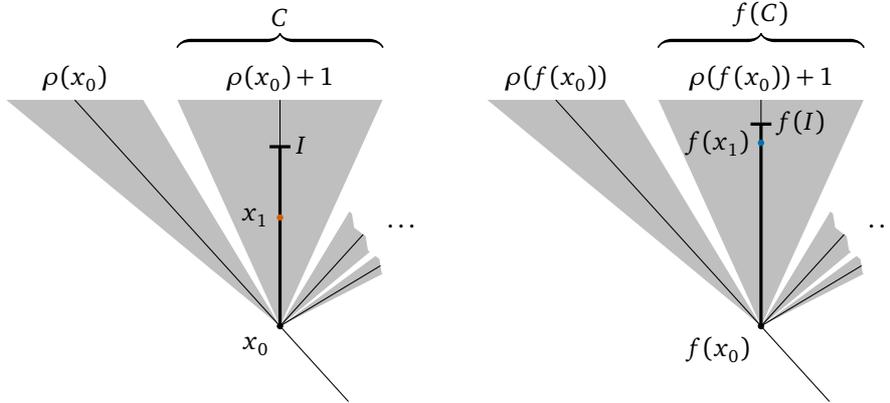
\begin{figure}
		\centering
		\begin{tikzpicture}
			\begin{scope}[xscale=0.9]
				\fill[lightgray] (-4,3) -- (0,0) -- (-2,3);
				\draw (1,-1) -- (-3,3) node [above] {$\rho(x_0)$};
				\fill[lightgray] (-1.5,3) -- (0,0) -- (1.5,3);
				\draw (0,0) -- (0,3) node [above] {$\rho(x_0) + 1$} node (t1) {};

				\begin{scope}
					\clip[decoration={zigzag, amplitude=1pt, segment length=10pt}, decorate] (0,-1) -- (2,0) -- (0,3) -- (-1,0) -- (0,-1);
					\fill[lightgray] (2,3) -- (0,0) -- (4,3);
					\draw (0,0) -- (3,3);
					\fill[lightgray] (4.5,3) -- (0,0) -- (6.5,3);
					\draw (0,0) -- (5.5,3);
				\end{scope}
				\node at (1.8,1.3) {$\cdots$};

				\draw[very thick,{-|}] (0,0) -- node[pos=1,right=2pt] {$I$} node[pos=0.6] (x1) {} ($0.8*(t1)$);
				\node[point,pallette_vermillion] at (x1) [label=left:$x_1$] {};

				\draw[brace] (-1.5,3.7) -- node[above=4pt] {$C$} (1.5,3.7);

				\node[point] at (0,0) [label=below left:$x_0$] {};
			\end{scope}
			\begin{scope}[xscale=0.9, xshift=200]
				\fill[lightgray] (-4,3) -- (0,0) -- (-2,3);
				\draw (1,-1) -- (-3,3) node [above] {$\rho(f(x_0))$};
				\fill[lightgray] (-1.5,3) -- (0,0) -- (1.5,3);
				\draw (0,0) -- (0,3) node [above] {$\rho(f(x_0)) + 1$} node (ft1) {};

				\begin{scope}
					\clip[decoration={zigzag, amplitude=1pt, segment length=10pt}, decorate] (0,-1) -- (2,0) -- (0,3) -- (-1,0) -- (0,-1);
					\fill[lightgray] (2,3) -- (0,0) -- (4,3);
					\draw (0,0) -- (3,3);
					\fill[lightgray] (4.5,3) -- (0,0) -- (6.5,3);
					\draw (0,0) -- (5.5,3);
				\end{scope}
				\node at (1.8,1.3) {$\cdots$};

				\draw[very thick,{-|}] (0,0) -- node[pos=1,right=2pt] {$f(I)$} node[pos=0.9] (fx1) {} ($0.9*(ft1)$);
				\node[point,pallette_blue] at (fx1) [label=left:$f(x_1)$] {};

				\draw[brace] (-1.5,3.7) -- node[above=4pt] {$f(C)$} (1.5,3.7);

				\node[point] at (0,0) [label=below left:$f(x_0)$] {};
			\end{scope}
		\end{tikzpicture}
		\caption{Picking $x_1$ when $\kappa \geq 3$. The grey regions represent connected components, while thin lines represent spines. The ordinal above a spine indicates its rank.}
		\label{fig:kappa geq 3; proof:strongly-uniform rigid}
	\end{figure}

	Now, consider the colouring $c$ along $I$ and along $f(I)$. By construction, these come from different colourings from $\A$. The order-isomorphism $f \restr I \colon I \to f(I)$ induces an isomorphism $U \to V$, which in turn induces an automorphism of $(0, \infty)$. Hence by the key property of the family $\A$ guaranteed by the Generic Colouring Lemma \ref{lem:generic family colourings}, there is $x_1 \in I$ coloured $c(x_1) = 1$ such that its image under $f$ is coloured $c(f(x_1)) = 0$. 

	Iterate this process to produce a sequence $x_0 < x_1 < \cdots$ such that $c(x_n) = 1$ and $c(f(x_n))=0$ for every $n > 0$. Then $\{x_0,x_1, \ldots\}$ determines a branch $B$ through $X_{\rho(x_0) + \omega}'$, while $\{f(x_0),f(x_1), \ldots\}$ determines a branch $f(B)$ through $X_{\rho(f(x_0)) + \omega}'$. Moreover, when deciding whether to extend $B$ at stage $\rho(x_0) + \omega$, we use a sequence whose tail is $(x_0,x_1, \ldots)$, and respectively for $f(B)$. But then the former branch gets extended, while the latter does not, contradicting that $f$ is an automorphism. \contradiction

	Let us now turn to the case $\kappa=2$. The issue here is that we can no longer guarantee that an initial segment of a new spine added above $x_0$ gets mapped to an initial segment of a new spine added above $f(x_0)$. Here's how we proceed. If the component $C$ above $x_0$ of rank $\rho(x_0)+1$ is mapped to the component above $f(x_0)$ of rank $\rho(f(x_0))+1$, proceed as before. Otherwise, the rank of $f(C)$ must be $\rho(f(x_0))$. See \cref{fig:kappa eq 2; proof:strongly-uniform rigid} for a representation of the situation. By \cref{lem:component facts} there is a coinitial interval $I \sse C$ of constant rank whose image $f(I) \sse f(C)$ has constant rank. So $I$ is an initial segment of a spine added above $x_0$ at stage $\rho(x_0)+1$ which maps onto a segment of a spine which already exists at stage $\rho(f(x_0))$. Now, using the key property of the colouring $\A$, pick any element $x_0'$ in the interior of $I$ with colour $c(x_0') = 1$ whose image has colour $c(f(x_0')) = 0$. Then we must have that the component $C'$ above $x_0'$ of rank $\rho(x_0')+1$ is mapped to the component above $f(x_0')$ of rank $\rho(f(x_0'))+1$, and so we can proceed as before to obtain $x_1 > x_0'$ coloured $c(x_1) = 1$ whose image is coloured $c(f(x_1))=0$. Iterating, we again build a sequence $x_0 < x_1' < x_1 < \cdots$ through $X$, all of whose elements (except possibly $x_0$) are coloured $1$ but with image coloured $0$. When deciding whether to extend the corresponding branch through $X_{\rho(x_0)+\omega}'$, we use the colours of a subsequence of $(x_0, x_1', x_1, \ldots)$. In the end, we find that this sequence has a limit in $X$, while its image does not. \contradiction
\end{proof}

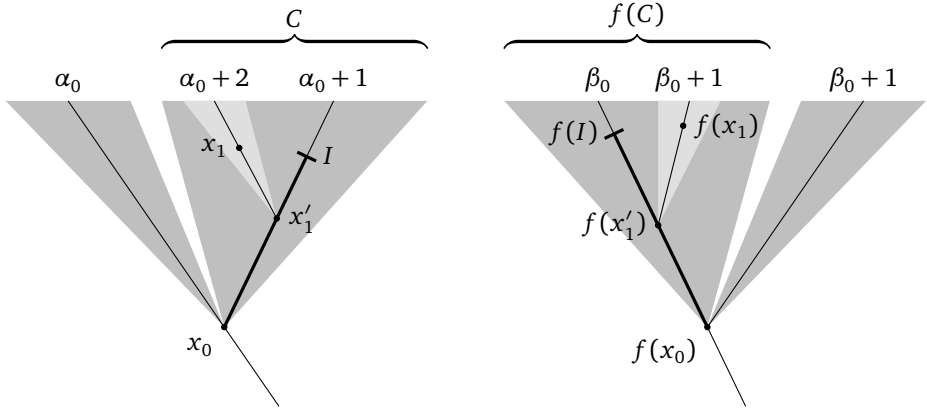
\begin{figure}
	\centering
	\begin{tikzpicture}[xscale=0.99]
		\begin{scope}[xscale=0.83]
			\fill[lightgray] (-1.5,3) -- (0,0) -- (-3.5,3);
			\draw (0.875,-1.05) -- (-2.5,3) node [above] {$\alpha_0$};
			\fill[lightgray] (-1,3) -- (0,0) -- (3.25,3);
			\draw (0,0) -- (1.75,3) node [above,xshift=0] {$\alpha_0 + 1$} node (t1) {};

			\path (0,0) -- node[pos=0.6, shape=coordinate] (x1p) {} ($0.8*(t1)$) node (ti) {};

			\fill[lightgray!50!white] ($(0,3)!(x1p)!(6,3) + (-1.5,0)$) -- (x1p) -- ($(0,3)!(x1p)!(6,3) + (-0.5,0)$);
			\draw (x1p) -- node[pos=0.6, shape=coordinate] (x1) {} ($(0,3)!(x1p)!(6,3) + (-1,0)$) node [above] {$\alpha_0 + 2$};
			\node[point] at (x1) [label=left:$x_1$] {};

			\draw[very thick,{-|}] (0,0) -- node[pos=1,right=2pt] {$I$} (ti);
			\node[point] at (x1p) [label=right:$x_1'$] {};

			\draw[brace] (-1,3.7) -- node[above=4pt] {$C$} (3.25,3.7);

			\node[point] at (0,0) [label=below left:$x_0$] {};
		\end{scope}
		\begin{scope}[xscale=0.83, xshift=220]
			\fill[lightgray] (1,3) -- (0,0) -- (-3.25,3);
			\draw (0.6125,-1.05) -- (-1.75,3) node [above, xshift=0] {$\beta_0$} node (ft1) {};
			\fill[lightgray] (1.5,3) -- (0,0) -- (3.5,3);
			\draw (0,0) -- (2.5,3) node [above] {$\beta_0 + 1$};

			\path (0,0) -- node[pos=0.5, shape=coordinate] (fx1p) {} ($0.9*(ft1)$) node (fti) {};

			\coordinate (ft3) at ($(0,3)!(fx1p)!(6,3) + (0.5,0)$);

			\fill[lightgray!50!white] ($(ft3) + (0.5,0)$) -- (fx1p) -- ($(ft3) + (-0.5,0)$);
			\draw (fx1p) -- node[pos=0.8, shape=coordinate] (fx1) {} (ft3) node [above] {$\beta_0 + 1$};
			\node[point] at (fx1) [label=right:$f(x_1)$] {};

			\draw[very thick,{-|}] (0,0) -- node[pos=1,left=2pt] {$f(I)$} (fti);
			\node[point] at (fx1p) [label=left:$f(x_1')$] {};

			\draw[brace] (-3.25,3.7) -- node[above=4pt] {$f(C)$} (1,3.7);

			\node[point] at (0,0) [label=below left:$f(x_0)$] {};
		\end{scope}
	\end{tikzpicture}
	\caption{Picking $x_1$ when $\kappa = 2$ and $\rho(f(C)) = \rho(f(x_0))$. The dark and light grey regions represent connected components, while thin lines represent spines. The ordinal above a spine indicates its rank, where $\alpha_0 = \rho(x_0)$ and $\beta_0=\rho(f(x_0))$.}
	\label{fig:kappa eq 2; proof:strongly-uniform rigid}
\end{figure}

With this \ZFC\ result established, let us examine once more the connection with a forcing notion. We again work with the $\kappa=\aleph_0$ case. We will define forcing conditions similarly to the discussion at the end of \cref{sec:weakly uniform}. This time however, instead of making promises that some points won't become branching, we promise that certain branches won't get ever get extended. There are a number of different ways of achieving this.
\begin{enumerate*}[label=(\arabic*)]
	\item The most direct way is to include in each condition a set of branches through the sub-branchwise-real tree, which are promised never to be extended. A technical issue which arises is that a branch through a smaller sub-branchwise-real tree need not be a branch through a larger one, since it need not be maximal as a chain any more. To make sure that the set of promises remains valid, we require that a stronger condition include a promise set which contains the unique extension of each promised branch in the weaker tree to the larger one.
	\item A more elegant method is to keep track instead of an order-preserving function from the approximations into \R. This way the branches which are not to be extended are exactly those on which the order-preserving function is unbounded.
	\item Yet a third way is make use of the universal continuously-gradable branchwise-real tree $U_\kappa$ defined in \cref{rem:cont grad not implies bounded rank}. Any maximal antichain $A$ through $U_\kappa$ defines a uniformly $\kappa$-branching branchwise-real tree: take $\{x \in U_\kappa \mid \exists a \in A \colon x < a\}$. We can then consider simply the set of all antichains in $U_\kappa$ under extension as a forcing poset.
\end{enumerate*}
For the following result, I remain agnostic as to the exact method used.

\begin{theorem}
	A generic tree $\ol X$ for the forcing just described is a uniformly $\kappa$-branching, rigid branchwise-real tree.
\end{theorem}

\begin{proof}[Proof sketch]
	The hard part is again showing that $\ol X$ is rigid. Take $\dot f \colon \ol X \to \ol X$ a non-trivial automorphism. By a countable-closure argument, we can find a condition $p_0$ forcing that $\dot f$ is a non-trivial order-preserving function whose domain contains the approximation $X_0$ at $p_0$ with $\dot f$ decided on every element of $X_0$, and with an element $x_0 \in X_0$ such that $\dot f(x_0)$ is incomparable with $x_0$. By another countable-closure argument, it is dense to find a condition $p_1 < p_0$ such that the domain of $\dot f$ contains the approximation $X_1$ at $p_1$ and such that there is $x_1 \in X_1$ not less than any element in $X_0$ with $x_0 < x_1$. Continuing in this way, we build a chain of forcing conditions $p_0 > p_1 > \cdots$ together with $x_0 < x_1 < \cdots$ in $\bigcup_{n \in \omega}X_n$. Let $p_\omega$ be the limit of the $p_n$'s. Then we can strengthen $p_\omega$ so that the branch determined by $\{x_0, x_1, \ldots\}$ gets extended, while we promise never to extend the branch determined by $\{\dot f(x_0), \dot f(x_1), \ldots\}$. (In the case where we keep track of an order-preserving function into \R, we would need to make a tweak to ensure that one sequence remains bounded while the other becomes unbounded.) Therefore, we can densely find a contradiction to the fact that $\dot f$ is order-preserving. \contradiction
\end{proof}

Turning back to our \ZFC\ constructions, the tree constructed in \cref{thm:strongly-uniform rigid} is continuously gradable. By suitably modifying the method, we can construct one which is not continuously gradable.

\begin{theorem}\label{thm:strongly-uniform rigid non gradable}
	Let $2 \leq \kappa \leq \cont^+$. There is a uniformly $\kappa$-branching, rigid branchwise-real tree which is not continuously gradable.
\end{theorem}

This time, the construction must proceed for $\omega_1$-many steps, because \cref{lem:countable height gradable} shows that any construction terminating after countably many steps would yield a continuously gradable tree. This then introduces a complication, as we need to ensure that there are no branches through the final tree which contain points of every rank below $\omega_1$. Indeed, such a branch would contain an $\omega_1$-sequence, violating condition \ref{item:interval; def:brot} that every branch is isomorphic to a real interval. Now, as mentioned in \cref{rem:cont grad not implies bounded rank}, not every tree constructed in $\omega_1$-many steps is non-continuously-gradable, so we also need to guarantee this. This is done by realising a known non-\R-gradable well-stratified tree inside the final branchwise-real tree, following techniques used in \cite{gradability-paper-published}. Lastly, these two competing requirements need to be balanced with the need for the final tree to be rigid. These desiderata are achieved following the method used in the proof of \cref{thm:strongly-uniform rigid} above, this time utilising $\omega$-gradings instead of $2$-gradings. At limit steps, the idea is essentially to extend a branch if and only if in the sequence of colours $(c_\alpha'(x_0), c_\alpha'(x_1), \ldots)$ of the maximal points of each rank, every element of $\omega$ appears only finitely often.

First, let us meet the non-\R-gradable well-stratified tree without uncountable branches which will be found inside the final branchwise-real tree.

\begin{lemma}\label{lem:R-ungrad tree}
	Let $T$ be the tree of functions $r$ into $\omega$, whose domain is a countable ordinal, such that when restricted to the set of successor ordinals $r$ is injective into $\omega \setminus \{0\}$, and elsewhere $r$ is identically $0$. We consider $T$ as a partial order under function extension. Then $T$ is a well-stratified tree without an \R-grading.
\end{lemma}

\begin{proof}
	The tree $T$ is isomorphic to the tree $\Shift(\In_\omega)$ defined in Section 3 of \cite{gradability-paper-published}. There, using a proof due to Baumgartner, Gavin and Laver \cite{baumthesis}, it is shown that this tree is well-stratified, has no \R-grading, and moreover that every branch is countable.
\end{proof}

With this to hand, we can now prove the theorem.

\begin{proof}[Proof of \cref{thm:strongly-uniform rigid non gradable}]
	Let $\A$ be a $\cont^+$-sized family of $(\omega \setminus \{0\})$-colourings of $(0,\infty)$ as per the Generic Colouring Lemma \ref{lem:generic family colourings}. We construct $X$ in stages $X_\alpha$ for $\alpha < \omega_1$, colouring each spine appearing at a stage which is the successor of a successor using a different colouring from $\A$. As in \cref{thm:strongly-uniform rigid}, these colourings are used to determine which branches to extend at limit stages. This will ensure that the resulting tree $X$ is rigid, that every branch is isomorphic to a real interval and that there is no continuous grading.

	Proceed as in \cref{thm:strongly-uniform rigid}. First partition $\A$ into $\omega_1$-many batches of size $\cont^+$, so that we never run out. We build each $X_\alpha$ and colouring $c_\alpha$, together with the rank function $\rho$, recursively. Start with $X_0$ the singleton coloured $0$. As before, at successor stages add $(\kappa - 1)$-many new open spines above the points added at the previous stage. If we are at the successor of $0$ or a limit, colour the new spines all with $0$, otherwise colour each using new colourings from $\A$. Now assume that $\alpha < \omega_1$ is a limit. First let $X_\alpha' \defeq \bigcup_{\beta<\alpha} X_\beta$ and $c_\alpha' \defeq \bigcup_{\beta<\alpha} c_\beta$. We need to decide which new branches through $X_\alpha'$ to extend. As before, a branch $B$ is determined by its sequence $(x_0, x_1, \ldots)$ of maximum points of each rank $\beta<\alpha$. Extend the branch $B$ if and only if no non-zero colour appears infinitely often in the sequence $(c_\alpha'(x_0), c_\alpha'(x_1), \ldots)$. To extend a branch, add a new point above it coloured $0$. Let $X_\alpha$ and $c_\alpha$ be result of performing this operation on every branch in $X_\alpha'$.

	Finally, let $X \defeq \bigcup_{\alpha < \omega_1} X_\alpha$ and $c \defeq \bigcup_{\alpha < \omega_1} c_\alpha$. Note that at each stage we made sure that every point is branching of degree $\kappa$, so $X$ is uniformly $\kappa$-branching. To see that $X$ is branchwise-real, take any branch $B$ through $X$. Then $B$ is determined by its sequence $(x_0, x_1, \ldots)$ of maximum points of each rank $\beta<\omega_1$. Consider the sequence $(c_\alpha(x_0), c_\alpha(x_1), \ldots)$. Every colour indexed by the successor of a successor is non-zero, and every non-zero colour appears at most countably many times. Hence the sequence must have countable length. Therefore, the branch $B$ cannot contain a subset isomorphic to $\omega_1$, so it must be isomorphic to a real interval, as required.

	To show that $X$ is rigid, we follow the method in \cref{thm:strongly-uniform rigid}. If $f \colon X \to X$ is a non-trivial automorphism, we can pick $x_0$ such that $f(x_0) \neq x_0$. Note that if we consider the sequence of points below $x_0$, respectively $f(x_0)$, maximal in each rank, then no non-zero colour appears infinitely often. So to determine whether a branch containing $x_0$ gets extended at a limit, it suffices to consider the sequence of colours of points maximal in each rank lying \emph{above} $x_0$, and likewise for $f(x_0)$. As in \cref{thm:strongly-uniform rigid} we can construct a sequence $x_0 < x_1 < \cdots$ recursively, this time ensuring that $c(x_n) = n$ while $c(f(x_n)) = 1$, for each $n>2$. Note that while any of $x_0, x_1, f(x_0), f(x_1)$ may be coloured $0$, none of the rest of the points or their images can be, since they all have ranks which are the successor of a successor. Then the sequence $(x_0, x_1, \ldots)$ has a limit in $X$, while its image does not.

	Finally, to see that $X$ has no \R-grading, we realise the tree $T$ given by \cref{lem:R-ungrad tree} as a subtree. For this, we recursively construct an embedding $\iota \colon T \to X$ such that for every $r \in T$:
	\begin{itemize}
		\item $\rho(\iota(r)) = \rank(r)$, and
		\item $c(\iota(r)) = r(\alpha)$ when $\rank(r) = \alpha+1$.
	\end{itemize}
	Start by sending the root $\es$ to the root. 

	Take $r \in T$ and assume that we have defined $\iota(r)$. There are two cases. If $\rank(r)$ is $0$ or a limit, then $r$ has exactly one immediate successor $s$. Moreover $\rho(\iota(r))$ is $0$ or a limit, so by the way we constructed $X$ there are $\kappa$-many spines with base $\iota(r)$ added at stage $\rank(r)+1$, all of which are coloured $0$. Let $\iota(s)$ be any point on any one of these spines. The other case is when $\rank(r)$ is a successor. Then $r$ has $\kappa$-many successors. Moreover, there are $\kappa$-many spines with base $\iota(r)$ added at stage $\rank(r)+1$, each of which are coloured according to a colouring from $\A$. In particular, every non-zero colour appears on every spine. Place each successor $s$ of $r$ on a different spine, in such a way that $c(\iota(s)) = s(\rank(r))$. 

	Now take $r \in T$ with $\rank(r)$ a limit, and assume that we have defined $\iota(s)$ for all $s < r$. Consider the sequence $(\iota(r \restr 0), \iota(r \restr 1), \ldots)$. This determines a branch through $X'_{\rank(r)}$. Moreover, when determining whether to extend this branch, the sequence of colours used is:
	\begin{equation*}
	 	(c(\iota(r \restr 0)), c(\iota(r \restr 1)), \ldots) = (0, r(0), r(1), \ldots, 0, r(\omega), r(\omega+1), \ldots)
	\end{equation*}
	By definition of $T$, no non-zero colour appears more than once in this sequence. Therefore, the branch is extended, meaning that the sequence $(\iota(r \restr 0), \iota(r \restr 1), \ldots)$ has a supremum, which we can set as the image of $r$.

	Putting it all together, we obtain an embedding $\iota \colon T \to X$. Any continuous grading of the latter would restrict to an \R-grading of the former, which by \cref{lem:R-ungrad tree} is impossible.
\end{proof}

%% file: parts/6-conclusion.tex

\section{Open questions}

\label{sec:open questions}

In this final section I present a number of questions left open by the preceding investigation.

The strongest result in this article is that there exists both a continuously gradable and a non-continuously-gradable uniformly $\kappa$-branching, rigid branchwise-real tree for $2 \leq \kappa \leq \cont^+$. This limit of $\cont^+$ comes from the Generic Colouring Lemma \ref{lem:generic family colourings}. It is natural to wonder if we could get past this limit using some alternative strategy.

\begin{question}\label{q:larger kappa}
	Does there exist a uniformly $\kappa$-branching, rigid branchwise-real tree for $\kappa > \cont^+$?
\end{question}

Now, it may be that such a question is beyond \ZFC, and that the best we can hope for is a consistency result. We might try expanding the Generic Colouring Lemma by forcing both a $2^\cont$-sized family of colourings and that $2^\cont > \cont^+$. Or alternatively we could try to force our rigid tree directly, for example by using a forcing notion similar to that presented in \cref{sec:strongly uniform}.

Taking a different tack, throughout the present article we have been aiming to eliminate all order-automorphisms. What if wanted more fine-grained control over the automorphism group? It is not hard to modify the methods presented here to yield a branchwise-real tree with automorphism group
\begin{enumerate*}[label=(\roman*)]
	\item the symmetries of $\kappa$ as a set,
	\item $\Z$,
	\item the order-automorphism group of the positive reals or
	\item various combinations of these.
\end{enumerate*}

\begin{question}
	What is the class of automorphism groups of (uniformly $\kappa$-branching) branchwise-real trees? 
\end{question}

Finally, we have seen four examples of uniformly $\kappa$-branching branchwise-real trees: the minimal $M_\kappa$, the universal $U_\kappa$, the rigid continuously gradable one and the rigid non-continuously-gradable one. None of these trees are isomorphic. Moreover, each comes with a rank function, the supremum of whose values is, respectively, $\omega$, $\omega_1$, $\omega^2$ and $\omega_1$. Clearly, any branchwise-real tree may admit many different rank functions, but is there some canonical way in which we might stratify uniformly $\kappa$-branching branchwise-real trees in terms of their `complexity'? One direction of investigation might be to first investigate how these tree embed into one another. We already know that $M_\kappa$ is a minimum and $U_\kappa$ a maximum, as least for the continuously gradable ones.

\begin{question}
	What is the nature of the class of uniformly $\kappa$-branching (continuously gradable) branchwise-real trees under embeddability?
\end{question}

Alternatively, we might consider the following isomorphism-invariant of branchwise-wise real trees $X$: what is the minimum supremum value of a rank function on $X$? For $M_\kappa$, this is $\omega$, and for any non-continuously gradable branchwise-real tree, this must be $\omega_1$ (by \cref{lem:countable height gradable}). It is not immediate that the uniformly $\kappa$-branching, rigid, continuously gradable branchwise-real tree constructed in \cref{sec:strongly uniform} has minimum supremum rank $\omega^2$; however it is not hard to see that its minimum cannot be $\omega$. This leads to the following questions.

\begin{question}
	Which ordinals arise as the minimum supremum value of a rank function on a uniformly $\kappa$-branching branchwise-real tree?
\end{question}

%% file: parts/7-acknowledgements.tex

\section{Acknowledgements}

I would like to thank my supervisor, Joel David Hamkins, for much helpful guidance and proof-reading throughout the process of doing this research and writing the present article. I also wish to thank the anonymous reviewer, whose suggestions led to many improvements of the paper. This work was supported by the EPSRC [studentship with project reference \emph{2271793}].